\documentclass[11pt]{amsart}

\usepackage{amssymb}
\usepackage{amsmath}
\usepackage{enumitem}
\usepackage{graphicx}
\usepackage{verbatim}
\usepackage{multirow}
\usepackage{color}
\usepackage{epstopdf}
\usepackage[all]{xy}
\usepackage{hyperref}
\usepackage{framed}

\usepackage[left=1.26in,right=1.26in,top=1.65in,bottom=1.3in]{geometry}
\numberwithin{equation}{section}
\numberwithin{figure}{section}

\newtheorem*{theorem*}{Theorem}

\newtheorem{theorem}{Theorem}[section]
\newtheorem*{question*}{Question}%[section]
\newtheorem{lemma}[theorem]{Lemma}%[section]
\newtheorem{definition}[theorem]{Definition}%[section]
\newtheorem{proposition}[theorem]{Proposition}
\newtheorem{corollary}[theorem]{Corollary}

\newtheorem{remark}[theorem]{Remark}%[section]
\def\diam{\mathop{\hbox{\rm diam}}}

\def \L{\Lambda}

\def \M{\mathcal{M}}

\newcommand{\C}{\mathcal{C}}

\newcommand{\RR}{\mathbb{R}}
\newcommand{\NN}{\mathbb{N}}
\newcommand{\ph}{\varphi}
\newcommand{\eps}{\varepsilon}

\newcommand{\htop}{h_\mathrm{top}}
\newcommand{\Vl}{V_\mathrm{loc}}
\newcommand{\phigeo}{\varphi^{\mathrm{geo}}}
\newcommand{\ulim}{\varlimsup}
\newcommand{\llim}{\varliminf}
\newcommand{\Zspan}{Z^{\mathrm{span}}}
\newcommand{\Zsep}{Z^{\mathrm{sep}}}
\newcommand{\Zper}{Z^{\mathrm{per}}}
\newcommand{\one}{\mathbf{1}}
\newcommand{\bpsi}{\overline{\psi}}

\DeclareMathOperator{\Int}{int}
\DeclareMathOperator{\Per}{Per}

\newcommand{\foot}[1]{\marginpar{\tiny\raggedright #1}}
\newcommand{\vc}[1]{\foot{VC: #1}}

\title
[Equilibrium states in dynamics via geometric measure theory]
{
Equilibrium states in dynamical systems via geometric measure theory
}
\author{Vaughn Climenhaga}
\address{Department of Mathematics \\ University of Houston \\ Houston, TX
77204, USA}
\email{climenha@math.uh.edu}
\urladdr{https://www.math.uh.edu/~climenha/}
\author{Yakov Pesin}
\address{Department of Mathematics \\ Pennsylvania State University \\
University Park, PA 16802, USA}
\email{pesin@math.psu.edu}
\urladdr{http://www.math.psu.edu/pesin/}
\author{Agnieszka Zelerowicz}
\address{Department of Mathematics \\ Pennsylvania State University \\
University Park, PA 16802, USA}
\email{axz157@psu.edu}

\begin{document}

\date{\today}

\begin{abstract}
Given a dynamical system with a uniformly hyperbolic (``chaotic'') attractor, the physically relevant Sinai--Ruelle--Bowen (SRB) measure can be obtained as the limit of the dynamical evolution of the leaf volume along local unstable manifolds.  We extend this geometric construction to the substantially broader class of equilibrium states corresponding to H\"older continuous potentials; these states arise naturally in statistical physics and play a crucial role in studying stochastic behavior of dynamical systems. The key step in our construction is to replace leaf volume with a reference measure that is obtained from a Carath\'eodory dimension structure via an analogue of the construction of Hausdorff measure.  In particular, we give a new proof of existence and uniqueness of equilibrium states that does not use standard techniques based on Markov partitions or the specification property; our approach can be applied to systems that do not have Markov partitions and do not satisfy the specification property.
\end{abstract}

\thanks{V.~C.\ is partially supported by NSF grants DMS-1362838 and
DMS-1554794. Ya.~P.\ and A.~Z.\ are partially supported by NSF grant
DMS-1400027.}

\maketitle

\setcounter{tocdepth}{1}
\tableofcontents

%\footnotetext{{\it Key words and phrases.} Scaled sequences, entropy,
%cocycles, Lyapunov exponents scaled Lyapunov exponents}

%\footnotetext{2010 {\it Mathematics Subject classification}: 28D20,
%37A35,37B40}

\section{Introduction}

\subsection{Systems with hyperbolic behavior}
A \emph{smooth dynamical system with discrete time} consists of a smooth manifold $M$ -- the \emph{phase space} -- and a diffeomorphism $f\colon M\to M$.  Each state of the system is represented by a point $x\in M$, whose \emph{orbit} $(f^n(x))_{n\in \mathbb{Z}}$ gives the time evolution of that state.  We are interested in the case when the dynamics of $f$ exhibit \emph{hyperbolic} behavior.  Roughly speaking, this means that orbits of nearby points separate exponentially quickly in either forward or backward time; if the phase space is compact, this leads to the phenomenon popularly known as `chaos'.

Hyperbolic behavior turns out to be quite common, and for such systems it is not feasible to make specific forecasts of a single trajectory far into the future, because small initial errors quickly grow large enough to spoil the prediction.  On the other hand, one may hope to make statistical predictions about the asymptotic behavior of orbits of $f$. A measurement of the system corresponds to a function $\ph\colon M\to\RR$; the sequence $\ph$, $\ph\circ f$, $\ph\circ f^2,\dots$ represents the same observation made at successive times.  When specific forecasts of $\ph\circ f^n$ are impossible, we can treat this sequence as a stochastic process and make predictions about its asymptotic behavior.  For a more complete discussion of this point of view, see \cite{ER85}, \cite[Chapter 1]{rM87}, and \cite{mV97}.

\subsection{Physical measures and equilibrium states}\label{sec:physical-es}
To fully describe the stochastic process $(\ph\circ f^n)_{n\in\mathbb{Z}}$, we need a probability measure $\mu$ on $M$ that represents the likelihood of finding the system in a given state at the present time. The measure $f_*\mu$ defined by 
$\int\ph\,d(f_*\mu)=\int\ph\circ f\,d\mu$ represents the distribution one unit of time into the future. An \emph{invariant} measure has $\mu=f_*\mu$, and hence $\mu=f_*^n\mu$ for all $n$, so the sequence of observations becomes a \emph{stationary} stochastic process.  

In this paper we will consider \emph{uniformly hyperbolic systems}, for which the tangent bundle admits an invariant splitting $TM = E^u \oplus E^s$ such that $E^u$ is uniformly expanded and $E^s$ uniformly contracted by $Df$; see \S\ref{sec:motivators} for examples and \S\ref{sec:hyp-sets} for a precise definition.  Such systems have an extremely large set of invariant measures; for example, standard results show that there are infinitely many periodic orbits, each supporting an atomic invariant measure. Thus one is led to the problem of selecting a distinguished measure, or class of measures, that is most dynamically significant.

Since we work on a smooth manifold, it would be natural to consider an invariant volume form on $M$, or at least an invariant measure that is absolutely continuous with respect to volume.  However, for dissipative systems such as the \emph{solenoid} described in \S\ref{sec:solenoid}, no such invariant measure exists, and one must instead look for a \emph{Sinai--Ruelle--Bowen (SRB) measure}, which we describe in \S\S\ref{sec:conditional}--\ref{sec:physical}.  Such a measure is absolutely continuous ``in the unstable direction'', which is enough to guarantee that it is \emph{physically relevant}; it describes the asymptotic statistical behavior of volume-typical trajectories.

SRB measures can be constructed via the following ``geometric approach'': let $m$ be normalized Lebesgue measure (volume) for some Riemannian metric on $M$; consider its forward iterates $f_*^n m$; then average the first $N$ of these and take a limit measure as $N\to\infty$.  

Another approach to SRB measures, which we recall in \S\ref{sec:es}, is via \emph{thermodynamic formalism}, which imports mathematical tools from equilibrium statistical physics in order to describe the behavior of large ensembles of trajectories.  This program began in the late 1950's, when Kolmogorov and Sinai introduced the concept of entropy into dynamical systems; see \cite{aK07} for a historical overview.  Given a \emph{potential function} $\ph\colon M\to \RR$, one studies the \emph{equilibrium states} associated to $\ph$, which are invariant measures that maximize the quantity $h_\mu(f) + \int\ph\,d\mu$, where $h_\mu$ denotes the Kolmogorov--Sinai entropy.\footnote{From the statistical physics point of view, the quantity $E_\mu := -(h_\mu(f) + \int\ph\,d\mu)$ is the free energy of the system, so that an equilibrium state minimizes the free energy; see \cite[\S1.6]{oS15} for more details.} The maximum value is called the \emph{topological pressure} of $\ph$ and denoted $P(\ph)$.

In the 1960s and 70s, it was shown by Sinai, Ruelle, and Bowen that for uniformly hyperbolic systems, every H\"older continuous potential has a unique equilibrium state (see Section \ref{sec:hyp-sets}). Applying this result to the particular case of the \emph{geometric potential}\footnote{Here the determinant is taken with respect to any orthonormal bases for $E^u(x)$ and $E^u(f(x))$.  If the map $f$ is of class of smoothness $C^{1+\alpha}$ for some $\alpha>0$, then one can show that $\phigeo(x)$ is H\"older continuous.} $\phigeo(x) = -\log |\det Df|_{E^u(x)}|$, one has $P(\phigeo)=0$ and the equilibrium state is the SRB measure described above; see \S\ref{sec:es}.

\subsection{Different approaches to constructing equilibrium states}
There are two main classical approaches to thermodynamic formalism. The first uses \emph{Markov partitions} of the manifold $M$; we recall the general idea in \S\ref{sec:Markov}.  The second uses the \emph{specification property}, which we overview in \S\ref{gibbs-property}.

The purpose of this paper is to describe a third ``geometric'' approach,  which was outlined above for SRB measures: produce an equilibrium state as a limiting measure of the averaged pushforwards of some \emph{reference measure}, which need not be invariant.   For the physical SRB measure, this reference measure was Lebesgue; to extend this approach to other equilibrium states, one must start by choosing a new reference measure.  The definition of this reference measure, and its motivation and consequences, is the primary goal of this paper, and our main result can be roughly stated as follows:

%\begin{itemize}\item 
\begin{framed}\noindent
\emph{For every H\"older continuous potential $\ph$, one can use the tools of geometric measure theory to define a reference measure $m_\ph$ for which the pushforwards $f_*^n m_\ph$ converge in average to the unique equilibrium state for 
$\ph$.}
\end{framed}
%\end{itemize}

A precise statement of the result is given in \S\ref{sec:main-result}.  An important motivation for this work is that the ``geometric''  approach can be applied to more general situations beyond the uniformly hyperbolic systems studied in this paper.  For example, the geometric approach was used in \cite{CDP} to construct SRB measures for some non-uniformly hyperbolic systems, and in \cite{CPZ} we use it to construct equilibrium states for some partially hyperbolic systems.  The first two approaches -- Markov partitions and specification -- have also been extended beyond uniform hyperbolicity (see \cite{CP17} for a survey of the literature), but the overall theory in this generality is still very far from being complete, so it seems worthwhile to add another tool by developing the geometric approach as well.

\subsection{Reference measures for general potentials}

In the geometric construction of the physical SRB measure, one can take the reference measure to be either Lebesgue measure $m$ on $M$ or Lebesgue measure $m_W$ on any \emph{local unstable leaf} $W=\Vl^u(x)$.  These leaves are $d^u$-dimensional submanifolds of $M$ that are tangent at each point to the \emph{unstable distribution} $E^u(x) \subset T_xM$; they are expanded by the dynamics of $f$ and have the property that $f(\Vl^u(x)) \supset \Vl^u(f(x))$ (see \S\ref{sec:hyp-sets} for more details).  Given a local unstable leaf $W=\Vl^u(x)$, we will write $m_W$ or $m_x^u$ for the leaf volume determined by the induced Riemannian metric.  This has the following key properties.
\begin{enumerate}
\item $m_W = m_x^u$ is a finite nonzero Borel measure on $W$.
\item If $W_1$ and $W_2$ are local unstable leaves with non-trivial intersection, then $m_{W_1}$ and $m_{W_2}$ agree on the overlap.
\item 
Under the dynamics of $f$, the leaf volumes scale by the rule
\begin{equation}\label{eqn:volume-scaling}
m_{f(x)}^u(A) = \int_{f^{-1} A} |\det Df|_{E^u(y)}| \,dm_x^u(y).
\end{equation}
\end{enumerate}
As mentioned in \S\ref{sec:physical-es},
the geometric potential $\phigeo(x) = -\log |\det Df|_{E^u(x)}|$ has $P(\phigeo)=0$, so the integrand in \eqref{eqn:volume-scaling} can be written as $e^{P(\phigeo) - \phigeo(y)}$.
In \S\ref{sec:car}, given a continuous potential $\varphi$ we will construct on every local unstable leaf $W=\Vl^u(x)$ a reference measure $m_x^\C$ satisfying similar properties to $m_x^u$, but with scaling rule\footnote{Note that in general $P(\ph)\neq 0$.}
\begin{equation}\label{eqn:scaling}
m_{f(x)}^\C(A) = \int_{f^{-1}A} e^{P(\ph)-\ph(y)} \,dm_x^\C(y).
\end{equation}
The superscript $\C$ is shorthand for a \emph{Carath\'eodory dimension structure} determined by the potential $\ph$ and a scale $r>0$; see \S\ref{car-struc} for the essential facts about such structures, and \cite{pes97} for a complete description. Roughly speaking, the definitions of $P(\ph)$ and $m_x^\C$ are analogous to the definitions of Hausdorff dimension and Hausdorff measure, respectively, but take the dynamics into account.
Recall that the latter definitions involve covers by balls of decreasing radius; the modification to obtain our quantities involves covering by dynamically defined balls, as explained in \S\ref{sec:car}.

\subsection{Some history}
The idea of constructing dynamically significant measures for uniformly hyperbolic maps by first finding measures on unstable leaves with certain scaling properties goes at least as far back as work of Sinai \cite{yS68}, which relies on Markov partitions. For uniformly hyperbolic dynamical systems with continuous time (flows) and the potential $\varphi=0$ the corresponding equilibrium state, which is the measure of maximal entropy, was obtained by Margulis \cite{gM70}; he used a different construction of leaf measures via functional analysis of a special operator (induced by the dynamics) acting on the Banach space of continuous functions with compact support on unstable leaves. These leaf measures were studied further in \cite{RS75,BM77}.

Hasselblatt gave a description of the Margulis measure in terms of Hausdorff dimension \cite{bH89}, generalizing a result obtained by Hamenst\"adt for geodesic flows on negatively curved compact manifolds \cite{uH89}. In this geometric setting, where stable and unstable leaves are naturally identified with the ideal boundary of the universal cover, Kaimanovich observed in \cite{vK90,vK91} that these leaf measures could be identified with the measures on the ideal boundary introduced by Patterson \cite{sP76} and Sullivan \cite{dS79}. For geodesic flows in negative curvature, this approach was recently extended to nonzero potentials by Paulin, Pollicott, and Schapira \cite{PPS15}.

For general hyperbolic systems and nonzero potential functions, families of leaf measures with the appropriate scaling properties were constructed by Haydn \cite{nH94} and Leplaideur \cite{Lep}, both using Markov partitions.  The key innovation in the present paper is that we can construct these leaf measures directly, without using Markov partitions, by an approach reminiscent of Hasselblatt's from \cite{bH89}. This requires us to interpret quantities in thermodynamic formalism by analogy with Hausdorff dimension, an idea which was introduced by Bowen for entropy \cite{rB73}, developed by Pesin and Pitskel' for pressure \cite{PP84}, and generalized further by Pesin \cite{yP88,pes97}.

\subsection{Plan of the paper}
We describe some motivating examples in \S\ref{sec:motivators}, and give general background definitions in \S\ref{sec:thermodynamics}.  These sections are addressed to a general mathematical audience, and the reader who is already familiar with thermodynamic formalism for hyperbolic dynamical systems can safely skip to \S\ref{sec:main-result}, where we give the new definition of the reference measures $m_x^\C$ and formulate our main results.  In \S\ref{car-struc} we recall the necessary results on Carath\'eodory dimension characteristics and describe some applications of our results to dimension theory.  
For well-known general results, we omit the proofs and give references to the literature where proofs can be found.  For the new results stated here, we give an outline of the proofs in \S\ref{sec:proofs}, and refer to \cite{CPZ} for complete details.

%The remaining sections give the proofs.  In \S\ref{sec:measures} we prove that the measures $m_x^\C$ are nonzero and finite, and in \S\ref{sec:scaling} we prove that they have the \emph{correct} behavior under iteration by $f$ and under stable holonomy (``sliding along stable manifolds''). Finally, in \S\ref{sec:geometric} we prove that the averaged pushforwards of $m_x^\C$ converge to the unique equilibrium state.

\subsection*{Acknowledgments}
This work had its genesis in workshops at ICERM (Brown University) and ESI (Vienna) in March and April 2016, respectively.  We are grateful to both institutions for their hospitality and for creating a productive scientific environment.

\section{Motivating examples}\label{sec:motivators}

Before recalling general definitions about uniformly hyperbolic systems and their invariant measures in \S\ref{sec:thermodynamics}, we describe three examples to motivate the idea of a `physical measure'.  Our discussion here is meant to convey the overall picture and omits many details.

\subsection{Hyperbolic toral automorphisms}\label{sec:hta}
Our first example is the diffeomorphism $f$ on the torus $\mathbb{T}^2 = \RR^2/\mathbb{Z}^2$ induced by the linear action of the matrix $L=\big(\begin{smallmatrix} 2 & 1 \\ 1 & 1 \end{smallmatrix}\big)$ on $\RR^2$, as shown in Figure \ref{fig:hta}.

\begin{figure}[htbp]
\includegraphics[width=.9\textwidth]{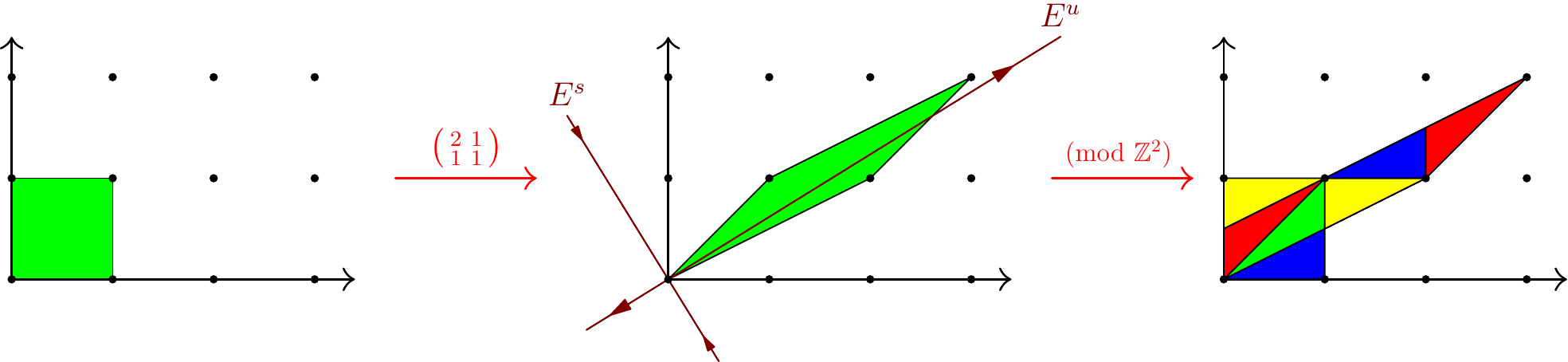}
\caption{Lebesgue measure is preserved by $f$.}
\label{fig:hta}
\end{figure}

This system is \emph{uniformly hyperbolic}:
The matrix $L$ has two positive real eigenvalues $\lambda_s < 1 < \lambda_u$, whose associated eigenspaces $E^s$ and $E^u$ give a $Df$-invariant splitting of the tangent bundle $T\mathbb{T}^2$.  The lines in $\RR^2$ parallel to these subspaces project to $f$-invariant foliations $W^s$ and $W^u$ of the torus.

What about invariant measures?
If $p\in \mathbb{T}^2$ has $f^n(p)=p$, then the measure
$\mu = \frac 1n(\delta_p + \delta_{f(p)} + \cdots + \delta_{f^{n-1}(p)})$ is invariant. Every point with rational coordinates is periodic for $f$, so this gives infinitely many $f$-invariant measures.  Lebesgue measure is also invariant since $\det Df = \det L = 1$. (This is far from a complete list, as we will see.) 

A measure $\mu$ is \emph{ergodic} if every $f$-invariant function (every
$\ph\in L^1(\mu)$ with $\ph =\ph\circ f$) is constant $\mu$-almost everywhere. One can check easily that the periodic orbit measures from above are ergodic, and with a little more work that Lebesgue measure is ergodic too.\footnote{This can be proved either by Fourier analysis or by the more geometric \emph{Hopf argument}, see \S\ref{sliding}.}  Birkhoff's ergodic theorem says that if $\mu$ is ergodic, then $\mu$-almost everywhere orbit has asymptotic behavior controlled by $\mu$.  More precisely, we say that the \emph{basin of attraction} for $\mu$ is the set of initial conditions satisfying a law of large numbers governed by $\mu$ for continuous observables:
\begin{equation}\label{eqn:basin}
B_\mu = \bigg\{x\in \mathbb{T}^2 : \frac 1n \sum_{k=0}^{n-1} \ph(f^k x) \xrightarrow{n\to\infty} \int \ph\,d\mu \text{ for all }\ph\in C(\mathbb{T}^2,\RR)\bigg\}.
\end{equation}
The ergodic theorem says that if $\mu$ is ergodic, then $\mu(B_\mu)=1$.

For the periodic orbit measures, this says very little, since it leaves open the possibility that the measure $\mu$ only controls the asymptotic behavior of finitely many orbits.\footnote{In fact $B_\mu$ is infinite, being a union of leaves of the stable foliation $W^s$.} For Lebesgue measure $m$, however, this says quite a lot: $m$ governs the statistical behavior of Lebesgue-almost every orbit, and in particular, a point chosen at random with respect to any volume form on $\mathbb{T}^2$ has a trajectory whose asymptotic behavior is controlled by $m$.  This is the sense in which Lebesgue measure is the `physically relevant' invariant measure, and we make the following definition.
\begin{definition}\label{def:physical}
An invariant measure $\mu$ for a diffeomorphism $f$ is a \emph{physical measure} if its basin $B_\mu$ has positive volume.
\end{definition}

\subsection{Smale--Williams solenoid}\label{sec:solenoid}
From Birkhoff's ergodic theorem, we see that if $\mu$ is an ergodic invariant measure that is equivalent to a volume form,\footnote{Recall that two measures  $\mu$ and $\nu$ are \emph{equivalent} if $\mu\ll\nu$ and $\nu\ll\mu$, in which case we write $\mu\sim\nu$.} then that volume form gives full weight to the basin $B_\mu$, and so a volume-typical trajectory has asymptotic behavior controlled by $\mu$.  

The problem now is that there are many examples for which no such $\mu$ exists. One such is the \emph{Smale--Williams solenoid} studied in \cite[\S I.9]{sS67} and \cite{rW67}; see also \cite[Lecture 29]{PC} for a gentle introduction and further discussion. This is a map from the open solid torus $U$ into itself. Abstractly, the solid torus is the direct product of a disc and a circle, so that one may use coordinates $(x,y,\theta)$ on $U$, where $x$ and $y$ give coordinates on the disc and $\theta$ is the angular coordinate on the circle. Define a map $f\colon U\to U$ by
\begin{equation}\label{eqn:solenoid}
f(x,y,\theta) := (\tfrac 14 x + \tfrac 12\cos\theta, \tfrac 14 y + \tfrac 12 \sin\theta, 2\theta).
\end{equation}
Figure \ref{fig:solenoid} shows two iterates of $f$, with half of the original torus for reference. 

\begin{figure}[htbp]
\includegraphics[width=.28\textwidth]{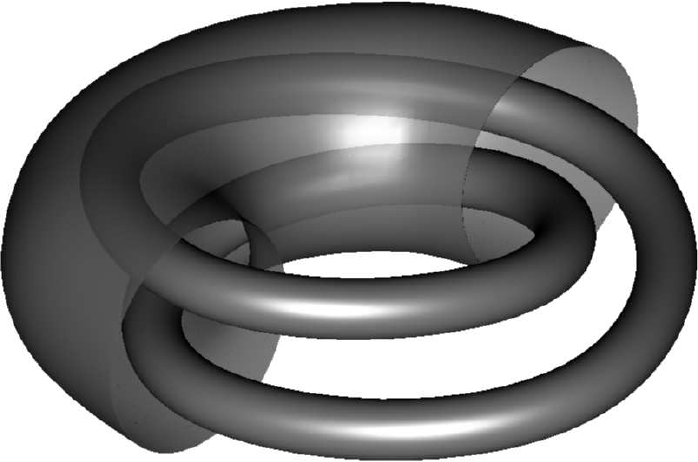}
\qquad\qquad
\includegraphics[width=.28\textwidth]{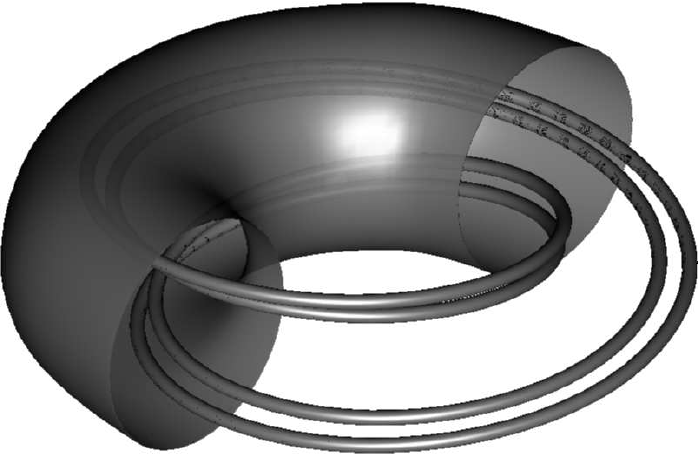}
\caption{No absolutely continuous invariant measure.}
\label{fig:solenoid}
\end{figure}

Every invariant measure is supported on the \emph{attractor} $\Lambda = \bigcap_{n\geq 0} f^n(U)$, which has zero volume.  In particular, there is no invariant measure that is absolutely continuous with respect to volume.   Nevertheless, it is still possible to find an invariant measure that is `physically relevant' in the sense given above.  To do this, first observe that since the solenoid map $f$ contracts distances along each cross-section $\mathbb{D}^2\times \{\theta_0\}$, any two points in the same cross-section have orbits with the same (forward) asymptotic behavior:
given an invariant measure $\mu$, the basin $B_\mu$ is a union of such cross-sections.

\begin{figure}[htbp]
\includegraphics[width=.6\textwidth]{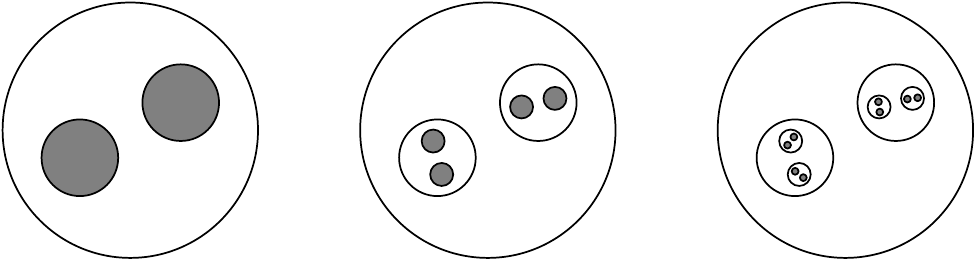}
\caption{A cross-section of the attractor.}
\label{fig:swscantor}
\end{figure}

This fact suggests that we should look for an invariant measure that is absolutely continuous in the direction of the circle coordinate $\theta$, which is expanded by $f$.  To construct such a measure, observe that each cross-section intersects the images $f^n(U)$ in a nested sequence of unions of discs, as shown in Figure \ref{fig:swscantor}, so the attractor $\Lambda$ intersects this cross-section in a Cantor set.  Thus $\Lambda$ is locally the direct product of an interval in the expanding direction and a Cantor set in the contracting directions. Let $m^u$ be Lebesgue measure on the circle, and let $\mu$ be the measure on $\Lambda$ that projects to $m^u$ and gives equal weight to each of the $2^n$ pieces at the $n$th level of the Cantor set construction in Figure \ref{fig:swscantor}. One can show without too much difficulty that $\mu$ is invariant and ergodic, and that moreover $B_\mu$ has full volume in the solid torus $U$.  Thus even though $\mu$ is singular, it is still the physically relevant invariant measure due to its absolute continuity in the expanding direction.

\subsection{Smale's horseshoe}\label{sec:horseshoe}
Finally, we recall an example for which no physical measure exists -- the \emph{horseshoe} introduced by Smale in the early 1960s; see \cite[\S I.5]{sS67} and \cite[Lecture 31]{PC} for more details, see also \cite{sS98} for more history.
Consider a map $f\colon R\to \mathbb{R}^2$ which acts on the square $R:=[0,1]^2$
as shown in Figure \ref{fig:horseshoe}: first the square is contracted vertically by a factor of $\alpha<1/2$ and stretched horizontally by a factor of $\beta> 2$;
then it is bent and positioned so that $f(R)\cap R$ consists of two rectangles of height $\alpha$ and length $1$.
 
\begin{figure}[htbp]
\includegraphics[width=.85\textwidth]{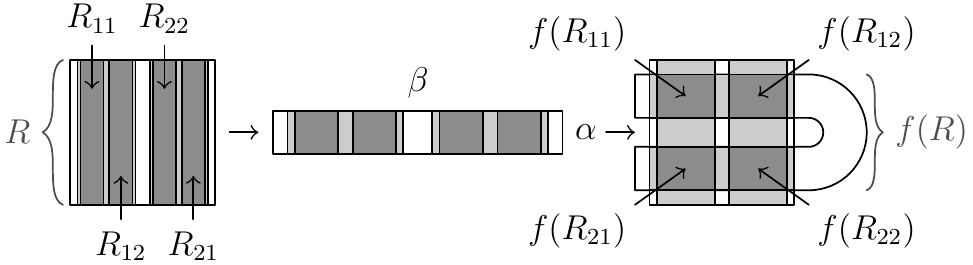}
\caption{No physical measure.}
\label{fig:horseshoe}
\end{figure}

Observe that a part of the square $R$ is mapped to the complement of $R$. Consequently, $f^2$ is not defined on the whole square $R$, but only on the union of two vertical strips in $R$.  The set where $f^3$ is defined is the union of four vertical strips, two inside each of the previous ones, and so on; there is a Cantor set $C^u \subset [0,1]$ such that every point outside $C^u\times [0,1]$ can be iterated only finitely many times before leaving $R$.  In particular, every $f$-invariant measure has $B_\mu \subset C^u\times [0,1]$, and hence $B_\mu$ is Lebesgue-null, so there is no physical measure.

Note that the argument in the previous paragraph did not consider the stable (vertical) direction at all.  For completeness, observe that there is a Cantor set $C^s\subset [0,1]$ such that $\bigcap_{n\geq 0} f^n(R) = [0,1]\times C^s$, and that the maximal $f$-invariant set $\Lambda:=  \bigcap_{n\in \mathbb{Z}} f^n(R)$ is a direct product $C^u\times C^s$.

\subsection{Main ideas}\label{sec:main-ideas}
The three examples discussed so far have certain features in common, which are representative of uniformly hyperbolic systems.  

First: every invariant measure lives on a compact invariant set $\Lambda$ that is locally the direct product of two sets, one contracted by the dynamics and one expanded.  For the hyperbolic toral automorphism $\big(\begin{smallmatrix} 2 & 1 \\ 1 & 1\end{smallmatrix}\big)$, $\Lambda=\mathbb{T}^2$ and both of these sets were intervals; for the solenoid, there was an interval in the expanding direction and a Cantor set in the contracting direction;  for the horseshoe, both were Cantor sets.

Second: the physically relevant invariant measure (when it existed) could also be expressed as a direct product.  For the hyperbolic toral automorphism, it was a product of Lebesgue measure on the two intervals.  For the solenoid, it was a product of Lebesgue measure on the interval (the expanding circle coordinate) and a \emph{$(\frac 12,\frac 12)$-Bernoulli measure} on the contracting Cantor set.

Third, and most crucially for our purposes: in identifying the physical measure, it is enough to look at how invariant measures behave \emph{along the expanding (unstable) direction}.  We will make this precise in \S\ref{sec:conditional} when we discuss \emph{conditional measures}, and this idea will motivate our main construction in \S\ref{sec:car} of reference measures associated to different potential functions.

Note that there is an asymmetry in the previous paragraph, because we privilege the unstable direction over the stable one.  This is because our notion of physical measure has to do with asymptotic time averages as $n\to+\infty$.  If we would instead consider the asymptotics as $n\to-\infty$, then the roles of stable and unstable objects would be reversed. We should also stress an important difference between the case when the invariant set $\Lambda$ is an attractor (as in the second example) and the case when it is a Cantor set (as in the third example): in the former case the trajectories that start near $\Lambda$ exhibit chaotic behavior for all time $t>0$ (the phenomenon known as \emph{persistent chaos}), while in the latter case the chaotic behavior occurs for a limited period of time whenever the trajectory passes by in a vicinity of $\Lambda$ (the phenomenon known as \emph{intermittent chaos}). 

\section{Equilibrium states and their relatives}\label{sec:thermodynamics}

\subsection{Hyperbolic sets}\label{sec:hyp-sets}
Now we make our discussion more precise and more general. We consider a smooth Riemannian manifold $M$ and a $C^{1+\alpha}$ diffeomorphism $f\colon M\to M$, and restrict our attention to the dynamics of $f$ on a \emph{locally maximal hyperbolic set}. We recall here the basic definition and most relevant properties, referring the reader to the book of Katok and Hasselblatt \cite[Chapter 6]{Kat} for a more complete account. In what follows it is useful to keep in mind the three examples discussed above.

A \emph{hyperbolic set} for $f$ is a compact set $\Lambda\subset M$ with $f(\Lambda)=\Lambda$ such that for every $x\in\L$, the tangent space admits a decomposition $T_xM=E^s(x)\oplus E^u(x)$ with the following properties.
\begin{enumerate}
\item The splitting is $Df$-invariant: $Df_x(E^\sigma(x)) = E^\sigma(fx)$ for $\sigma=s,u$.
\item The stable subspace $E^s(x)$ is uniformly contracting and the unstable subspace $E^u(x)$ is uniformly expanding: there are constants $C\geq 1$ and $0<\chi<1$ such that for every $n\geq 0$ and $v^{s,u}\in E^{s,u}(x)$, we have
\[
\|Df^nv^s\|\leq C\chi^n\|v^s\| \quad\text{and}\quad
\|Df^{-n}v^u\|\leq C\chi^n\|v^u\|.
\]
Replacing the original Riemannian metric with an \emph{adapted metric},\footnote{This metric may not be smooth, but will be at least $C^{1+\gamma}$ for some $\gamma>0$, which is sufficient for our purposes.} we can (and will) take $C=1$.
\end{enumerate}
In the case $\Lambda=M$ the map $f$ is called an \emph{Anosov diffeomorphism}.

Of course there are some diffeomorphisms that do not have any hyperbolic sets (think of isometries), but it turns out that a very large class of diffeomorphisms do, including the examples from the previous section.  These examples also have the property that 
$\Lambda$ is \emph{locally maximal}, meaning that there is an open set 
$U\supset \Lambda$ for which any invariant set $\Lambda'\subset U$ is contained in 
$\Lambda$; in other words, $\Lambda = \bigcap_{n\in \mathbb{Z}} f^n(U)$.  In this case every $C^1$-perturbation of $f$ also has a locally maximal hyperbolic set contained in $U$; in particular, the set of diffeomorphisms possessing a locally maximal hyperbolic set is open in the $C^1$-topology.

A number of properties follow from the definition of a hyperbolic set.  First, the subspaces $E^{s,u}(x)$ depend continuously on $x\in\L$; in particular, the angle between them is uniformly away from zero. In fact, since $f$ is $C^{1+\alpha}$, the dependence on $x$ is H\"older continuous:
\begin{equation}\label{holder}
\rho(E^{s,u}(x),E^{s,u}(y))\le K d(x,y)^\beta,
\end{equation} 
where $\rho$ is the Grassmannian distance between the subspaces, $d$ is the
distance in $M$ generated by the (adapted) Riemannian metric, and $K, \beta>0$.

\begin{proposition}[{\cite[Theorem 6.2.3]{Kat}}]\label{prop:local-mfds}
The subspaces $E^{s,u}$ can be ``integrated'' locally: for every $x\in\L$ there exist \emph{local stable and unstable} submanifolds $\Vl^{s,u}(x)$ given via the graphs of 
$C^{1+\alpha}$ functions $\psi_x^{s,u}\colon B_x^{s,u}(0,\tau)\to E^{u,s}(x)$,\footnote{Here $B_x^{s,u}(0,\tau)$ is the ball in $E^{s,u}(x) \subset T_x M$ of radius $\tau$ centered at $0$.} for which we have:
\begin{enumerate}[label=\textup{(\arabic{*})}]
\item $\Vl^{s,u}(x)=\exp_x\{v+\psi^{s,u}_x(v): v\in B_x^{s,u}(0,\tau)\}$;
\item $x\in \Vl^{s,u}(x)$ and $T_x\Vl^{s,u}(x)=E^{s,u}(x)$;
\item $f(\Vl^u(x)) \supset \Vl^u(f(x))$ and $f(\Vl^s(x)) \subset \Vl^s(f(x))$;
\item\label{leaves-contract} there is $\lambda \in (\chi, 1)$ such that 
$d(f(y),f(z))\le\lambda d(y,z)$ for all $y,z\in \Vl^s(x)$ and 
$d(f^{-1}(y),f^{-1}(z))\le\lambda d(y,z)$ for all $y,z\in \Vl^u(x)$;\footnote{This means that the local unstable manifold for $f$ is the local stable manifold for $f^{-1}$.}
\item\label{unif-holder} there is $C>0$ such that the H\"older semi-norm satisfies $|D\psi^{s,u}_x|_\alpha\leq C$ for all $x\in\Lambda$.
\end{enumerate}
\end{proposition}
The number $\tau$ is the \emph{size} of the local manifolds and will be fixed at a sufficiently small value to guarantee various estimates (such as the last item in the list above); note that the properties listed above remain true if $\tau$ is decreased. The manifolds $\Vl^{s,u}(x)$ depend continuously on $x\in\L$.

Given a hyperbolic set $\Lambda$, there is $\eps>0$ such that for every $x,y\in\Lambda$ with $d(x,y) < \eps$,  the intersection $\Vl^s(x)\cap \Vl^u(y)$ consists of a single point, denoted by $[x,y]$ and called the \emph{Smale bracket} of $x$ and $y$.  One can show that $\Lambda$ is locally maximal if and only if $[x,y]\in \Lambda$ for all such $x,y$; this is the local product structure referred to in \S\ref{sec:main-ideas}.

\begin{definition}
A closed set $R\subset \Lambda$ is called a \emph{rectangle} if $[x,y]$ is defined and lies in $R$ for all $x,y\in R$.  Given $p\in R$ we write $V_R^{s,u}(p) = \Vl^{s,u}(p) \cap R$ for the parts of the local manifolds that lie in $R$.
\end{definition}

\begin{figure}[htbp]
\includegraphics[width=.4\textwidth]{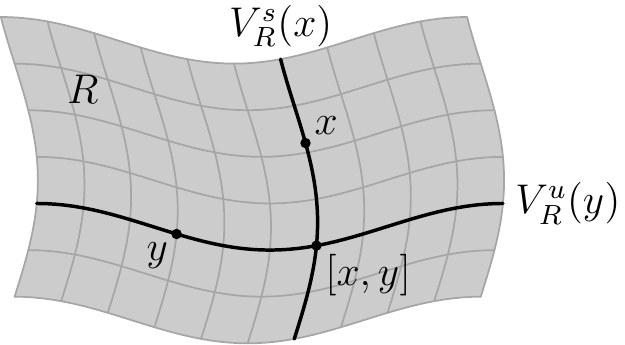}
\caption{A rectangle in the case when $\L=M$.}
\label{fig:rectangle}
\end{figure}

Given a rectangle $R$ and a point $p\in R$, let $A = V_R^u(p) \subset \Vl^u(p)$ and $B=V_R^s(p) \subset \Vl^s(p)$. Then $[x,y]$ is defined for all $x\in A$ and $y\in B$, and  
\begin{equation}\label{rectangle}
R = [A,B] := \{ [x,y] : x\in A, y\in B\}.
\end{equation}
Conversely, it is not hard to show that \eqref{rectangle} defines a rectangle whenever 
$p\in \Lambda$ and the closed sets $A\subset \Vl^u(p)\cap \Lambda$ and 
$B\subset \Vl^s(p)\cap \Lambda$ are contained in a sufficiently small neighborhood of $p$.

For the hyperbolic toral automorphism from \S\ref{sec:hta},  we can take $A$ and $B$ to be intervals around $p$ in the stable and unstable directions respectively; then $[A,B]$ is the direct product of two intervals, consistent with our usual picture of a rectangle.  However, in general, we could just as easily let $A$ and $B$ be Cantor sets, and thus obtain a dynamical rectangle that does not look like the picture we are familiar with.  For the solenoid and horseshoe, this is the only option; in these examples the hyperbolic set $\Lambda$ has zero volume and empty interior, and we see that rectangles are not even connected.

Indeed, there is a general dichotomy: given a $C^{1+\alpha}$ diffeomorphism $f$ and a locally maximal hyperbolic set $\Lambda$, we either have $\Lambda=M$ (in which case $f$ is an Anosov diffeomorphism), or $\Lambda$ has zero volume.\footnote{This dichotomy can fail if $f$ is only $C^1$; see \cite{rB75b}.} Even when $m(\Lambda)=0$, the dynamics on 
$\Lambda$ still influences the behavior of nearby trajectories, as is most apparent when 
$\Lambda$ is an \emph{attractor}, meaning that there is an open set $U\supset \Lambda$ (called a \emph{trapping region}) such that $\overline{f(U)}\subset U$ and 
$\Lambda=\bigcap_{n\in\mathbb{N}}f^n(U)$, as was the case for the solenoid. In this case every trajectory that enters $U$ is \emph{shadowed} by some trajectory in $\Lambda$, and $\Lambda$ is a union of unstable manifolds: $\Vl^u(x) \subset \Lambda$ for every $x\in\Lambda$.

One final comment on the topological dynamics of hyperbolic sets is in order.  Recall that if $X$ is a compact metric space and $f\colon X\to X$ is continuous, then the system $(X,f)$ is called \emph{topologically transitive} if for every open sets $U,V\subset X$ there is $n\in \NN$ such that $U\cap f^{-n}(V) \neq\emptyset$, and \emph{topologically mixing} if for every such $U,V$ there is $N\in \NN$ such that $U\cap f^{-n}(V)\neq\emptyset$ for all $n\geq N$.  Every locally maximal hyperbolic set $\Lambda$ admits a \emph{spectral decomposition} \cite{sS67}: it can be written as a union of disjoint closed invariant subsets $\Lambda_1,\dots,\Lambda_k\subset \Lambda$ such that each $f|_{\Lambda_i}$ is topologically transitive, and moreover each $\Lambda_i$ is a union of disjoint closed invariant subsets $\Lambda_{i,1},\dots, \Lambda_{i,n_i}$ such that 
$f(\Lambda_{i,j}) = \Lambda_{i,j+1}$ for $1\leq j < n_i$, 
$f(\Lambda_{i,n_i}) = \Lambda_{i,1}$, and each $f^{n_i}|_{\Lambda_{i,j}}$ is topologically mixing.  For this reason there is no real loss of generality in restricting our attention to topologically mixing locally maximal hyperbolic sets.

\subsection{Conditional measures}\label{sec:conditional}
Now we consider measures on $\Lambda$, writing $\M(\Lambda)$ for the set of all Borel probability measures on $\Lambda$, and $\M(f,\Lambda)$ for the set of all such measures that are $f$-invariant.
We briefly mention several basic facts that play an important role in the proofs (see \cite[Chapter 4]{EW11} for details): the set $\M(f,\L)$ is convex, and its extreme points are precisely the ergodic measures; every $\mu\in \M(f,\L)$ has a unique \emph{ergodic decomposition} $\mu = \int_{\M^e(f,\L)} \nu \,d\zeta(\nu)$, where $\zeta$ is a probability measure on the space of ergodic measures $\M^e(f,\L)$; and finally, $\M(f,\L)$ is compact in the weak* topology.

As suggested by the discussion in  \S\ref{sec:main-ideas},  in order to understand how an invariant measure $\mu$ governs the forward asymptotic behavior of trajectories, we should study how $\mu$ behaves ``along the unstable direction''.  To make this precise, we now recall the notion of \emph{conditional measures}; for more details, see \cite{Roh} or \cite[\S5.3]{EW11}.

Given $\mu\in \M(f,\Lambda)$, consider a rectangle $R \subset \Lambda$ with $\mu(R)>0$. Let $\xi$ be the partition of $R$ by local unstable sets $V_R^u(x) = \Vl^u(x) \cap R$, $x\in R$; these depend continuously on the point $x$, so the partition $\xi$ is \emph{measurable}.
This implies that the measure $\mu$ can be disintegrated with respect to $\xi$: for $\mu$-almost every $x\in R$, there is a \emph{conditional measure} $\mu_{V_R^u(x)}^\xi$ on the partition element $V_R^u(x)$ for any Borel subset $E\subset R$, we have\footnote{For a finite partition, the obvious way to define a conditional measure on a partition element $A$ with $\mu(A)>0$ is to put $\mu_A(E) = \mu(E\cap A)/\mu(A)$.  Roughly speaking, measurability of $\xi$ guarantees that it can be written as a limit of finite partitions, and the conditional measures in \eqref{gibbs-split} are the limits of the conditional measures for the finite partitions; see Proposition \ref{prop:conditionals} for a precise statement.}
\begin{equation}\label{gibbs-split}
\mu(E)=\int_{R}\int_{V_R^u(x)}\one_{E}(y)\,d\mu_{V_R^u(x)}^\xi(y)\,d \mu(x).   
\end{equation}
Since $\mu_{V_R^u(x)}^\xi = \mu_{V_R^u(x')}^\xi$ whenever $x'\in V_R^u(x)$, the outer integral in \eqref{gibbs-split} can also be written as an integral over the quotient space $R/\xi$, which inherits a factor measure $\tilde\mu$ from $\mu|_R$ in the natural way.
By the local product structure of $R$, we can also fix $p\in R$ and identify $R/\xi$ with $V_R^s(p)$.  Then $\tilde\mu$ gives a measure on $V_R^s(p)$ by $\tilde\mu(A) = \mu(\bigcup_{x\in A} V_R^u(x))$.  Writing $\mu_x^u = \mu_{V_R^u(x)}^\xi$ for the conditional measure on the leaf through $x$, we can rewrite \eqref{gibbs-split} as\footnote{Note that $\mu_x^u$ depends on $R$, although this is suppressed in the notation.}
\begin{equation}\label{gibbs-split2}
\mu(E) = \int_{V_R^s(p)} \int_{V_R^u(x)} \one_E(y) \,d\mu_x^u(y) \,d\tilde\mu(x).
\end{equation}
This disintegration is unique under the assumption that the conditional measures are normalized.  Although the definition depends on $R$, in fact choosing a different rectangle $R'$ merely has the effect of multiplying $\mu_x^u$ by a constant factor on $R\cap R'$ \cite[Lemma 2.5]{CPZ}.

 One can similarly define a system of conditional measures $\{\mu^s_x\}$ on $V_R^s(x)$ for $x\in V_R^u(p)$;
it is then natural to ask whether the conditional measure $\mu_p^s$ agrees with the measure $\tilde\mu$ on $V_R^s(p)$ from \eqref{gibbs-split2}, and we will return to this question in \S\ref{sliding} when we discuss \emph{absolute continuity} and the \emph{Hopf argument}.

\subsection{SRB measures}\label{sec:physical}
Now suppose that $\Lambda$ is a hyperbolic attractor and hence, contains the local unstable leaf $\Vl^u(x)$ for every $x\in \Lambda$.

\begin{definition}\label{def:SRB}
Given a hyperbolic attractor $\Lambda$ for $f$ and a point $y\in \Lambda$ with a local unstable leaf $W=\Vl^u(y)$, let $m_W$ be the leaf volume on $W$ generated by the restriction of the Riemannian metric to $W$. An invariant measure $\mu$ is a \emph{Sinai--Ruelle--Bowen (SRB) measure} if for every rectangle $R\subset \Lambda$ with $\mu(R)>0$, the conditional measures $\mu_x^u$ are absolutely continuous with respect to the leaf volumes $m_x^u$ for $\tilde\mu$-almost every $x$.
\end{definition}

One of the major goals in the study of systems with some hyperbolicity is to construct SRB measures.  In the uniformly hyperbolic setting, this was done by Sinai, Ruelle, and Bowen.

\begin{theorem}[\cite{yS68,Bow,dR76}]\label{srb}
Let $\Lambda$ be a topologically transitive hyperbolic attractor for a $C^{1+\alpha}$ diffeomorphism $f$.  Then there is a unique SRB measure for $f|_\Lambda$.
\end{theorem}
As suggested by the discussion in \S\ref{sec:motivators}, it is not hard to show that SRB measures are \emph{physical} in the sense of Definition \ref{def:physical}.  In fact, one can prove that for hyperbolic attractors, SRB measures are the \emph{only} physical measures.\footnote{An example due to Bowen and Katok \cite[\S0.3]{aK80} shows that when $\Lambda$ is not an attractor, it can support a physical measure that is not SRB.}

In addition to this physicality property, it was shown in \cite{yS68,dR76} that the SRB measure $\mu$ has the property that
\begin{equation}\label{eqn:push-volume}
\mu = \lim_{n\to\infty} \frac 1n \sum_{k=0}^{n-1} f_*^k m|_U,
\end{equation}
where $m|_U$ is normalized volume on the trapping region $U \supset \Lambda$.\footnote{In fact, they proved the stronger property that $f_*^n m|_U \to \mu$.}  In \cite{PS82}, this idea was used in order to \emph{construct} SRB measures\footnote{More precisely, \cite{PS82} considered the partially hyperbolic setting and used this approach to construct invariant measures that are absolutely continuous along unstable leaves; SRB measures are a special case of this when the center bundle is trivial.} with $m|_U$ replaced by leaf volume $m_{\Vl^u(x)}$.  We refer to this as the ``geometric construction'' of SRB measures, and will return to it when we discuss our main results.  First, though, we observe that the original constructions of SRB measures followed a different approach and used mathematical tools borrowed from statistical physics, as we discuss in the next section.

\subsection{Equilibrium states}\label{sec:es}
It turns out that it is possible to relate the absolute continuity requirement in Definition \ref{def:SRB} to a variational problem. The Margulis--Ruelle inequality \cite{dR78} (see also \cite[\S9.3.2]{BP13}) states that for any invariant Borel measure $\mu$ supported on a hyperbolic set $\Lambda$,\footnote{There is a more general version of this inequality that holds without the assumption that $\mu$ is supported on a hyperbolic set, but it requires the notion of \emph{Lyapunov exponents}, which are beyond the scope of this paper.} we have the following upper bound for the Kolmogorov--Sinai entropy:
\begin{equation}\label{eqn:Margulis-Ruelle}
h_\mu(f) \leq \int \log|\det Df|_{E^u(x)}| \,d\mu.
\end{equation} 
Recall that $h_\mu(f)$ can be interpreted as the average asymptotic rate at which information is gained if we observe a stochastic process distributed according to $\mu$;  \eqref{eqn:Margulis-Ruelle} says that this rate can never exceed the average rate of expansion in the unstable direction. 

Pesin's entropy formula \cite{yP77} states that equality holds in \eqref{eqn:Margulis-Ruelle} if $\mu$ is absolutely continuous with respect to volume. In fact, Ledrappier and Strelcyn proved that it is sufficient for $\mu$ to have conditional measures on local unstable manifolds that are absolutely continuous with respect to leaf volume \cite{LS82}, and Ledrappier proved that this condition is also necessary \cite{fL84}. In other words, equality holds in \eqref{eqn:Margulis-Ruelle} if and only if $\mu$ is an SRB measure. 

Since every hyperbolic attractor $\Lambda$ has an SRB measure, we conclude that the function $\phigeo(x)=-\log|\det Df|_{E^u(x)}|$ has the property that
\[
\sup_{\mu\in\M(f,\Lambda)} \bigg( h_\mu(f) + \int \phigeo\,d\mu\bigg) = 0,
\]
and the SRB measure for $f$ is the unique measure achieving the supremum, as claimed in \S\ref{sec:physical-es}.  More generally, we have the following definition.

\begin{definition}\label{def:equilibrium}
Let $\ph\colon M\to\RR$ be a continuous function, which we call a  \emph{potential}. An \emph{equilbrium state} (or \emph{equilibrium measure}) for $\ph$ is a measure $\mu$ achieving the supremum
\begin{equation}\label{eqn:vp}
\sup_{\mu\in\M(f,\Lambda)} \bigg( h_\mu(f) + \int \ph\,d\mu\bigg).
\end{equation}
\end{definition}

Thus SRB measures are equilibrium states for the \emph{geometric potential} 
$\phigeo$, which is H\"older continuous on every hyperbolic set as long as $f$ is $C^{1+\alpha}$, by \eqref{holder}.  This means that existence and uniqueness of SRB measures is a special case of the following classical result.

\begin{theorem}[\cite{Sin,Bow,Rue}]\label{thm:es}
Let $\Lambda$ be a locally maximal hyperbolic set for a $C^{1+\alpha}$ diffeomorphism $f$ and $\ph\colon\Lambda\to\RR$ a H\"older continuous potential. Assume that 
$f|_\Lambda$ is topologically transitive. Then there exists a unique equilibrium state $\mu$ for $\ph$.
\end{theorem}

In \S\S\ref{sec:Markov}--\ref{gibbs-property} we briefly recall two classical proofs of Theorem \ref{thm:es} which are based on either symbolic representation of $f|_\Lambda$ as a topological Markov chain or on the specification property of $f|_\Lambda$.  In \S\ref{sec:car} we introduce the tools that we will use to provide a new proof which is based on some constructions in geometric measure theory. 

The function $\mu\mapsto h_\mu(f) + \int\ph\,d\mu$ is affine; it follows that the unique equilibrium state $\mu$ must be ergodic, otherwise every element of its ergodic decomposition would also be an equilibrium state.
In fact, it has many \emph{good} ergodic properties: one can prove that it is Bernoulli, has exponential decay of correlations, and satisfies the Central Limit Theorem \cite{Bow}.

The fundamental result of thermodynamic formalism is the \emph{variational principle}, which establishes that the supremum in \eqref{eqn:vp} is equal to the \emph{topological pressure} of $\ph$, which can be defined as follows without reference to invariant measures.

\begin{definition}\label{def:pressure}
Given an integer $n\geq 0$, consider the \emph{dynamical metric of order $n$}
\begin{equation}\label{eqn:dn}
d_n(x,y) = \max\{d(f^kx,f^ky) : 0\leq k< n\}
\end{equation}
and the associated \emph{Bowen balls} $B_n(x,r) = \{y : d_n(x,y) < r\}$ for each $r>0$. We say that $E\subset\Lambda$ is \emph{$(n,r)$-separated} if $d_n(x,y)\geq r$ for all
$x\neq y\in E$, and that $E$ is \emph{$(n,r)$-spanning} for $X\subset \Lambda$ if
$X\subset \bigcup_{x\in E} B_n(x,r)$.

Writing $S_n\ph(x)=\sum_{k=0}^{n-1} \ph(f^k x)$ for the $n$th \emph{Birkhoff sum} along the orbit of $x$, the \emph{partition sum} of $\ph$ on a set $X\subset \Lambda$ refers to one of the following two quantities:
\begin{align*}
\Zspan_n(X,\ph,r) &:= \inf \Big\{ \sum_{x\in E} e^{S_n\ph(x)} : E\subset X
\text{ is $(n,r)$-spanning for }X\Big\}, \\
\Zsep_n(X,\ph,r) &:= \sup \Big\{\sum_{x\in E} e^{S_n\ph(x)} : E\subset X \text{
is $(n,r)$-separated}\Big\}.
\end{align*}
Then the topological pressure is given by\footnote{The fact that the limits coincide is given by an elementary argument comparing $\Zspan_n$ and $\Zsep_n$. In fact, the limit in $r$ can be removed due to expansivity of $f|_\Lambda$; see Definition \ref{def:expansive} and %Lemma \ref{lem:P-on-W}
\cite[Theorem 9.6]{pW82}.}
\begin{equation}\label{eqn:pressure}
P(\ph) = \lim_{r\to 0} \ulim_{n\to\infty} \frac 1n \log \Zspan_n(\Lambda,\ph,r)
= \lim_{r\to 0} \ulim_{n\to\infty} \frac 1n \log\Zsep_n(\Lambda,\ph,r).
\end{equation}
(One gets the same value if $\ulim$ is replaced by $\llim$.)
\end{definition}

It is worth noting at this point that the definition of $P(\ph)$ bears a certain similarity to the definition of box dimension: one covers $X$ by a collection of balls at a given scale, associates a certain weight to this collection, and then computes the growth rate of this weight as the balls in the cover are refined.  The difference is that here the refinement is done dynamically rather than statically, and different balls carry different weight according to the ergodic sum $S_n\ph(x)$; we will discuss this point further in \S\ref{sec:car} and \S\ref{car-struc}.  When $\ph=0$, we obtain the \emph{topological entropy} $\htop(f) = P(0)$, which gives the asymptotic growth rate of the cardinality of an $(n,r)$-spanning or $(n,r)$-separated set; one can show that this is also the asymptotic growth rate of the number of periodic orbits in $\L$ of length $n$.

Now the variational principle \cite[Theorem 9.10]{pW82} can be stated as follows:
\begin{equation}\label{eqn:vp2}
P(\ph) = \sup_{\mu\in\M(f,\Lambda)} \bigg( h_\mu(f) + \int \ph\,d\mu\bigg).
\end{equation}
The discussion at the beginning of this section shows that $P(\phigeo)=0$.  Given a potential $\ph$, we see that an equilibrium state for $\ph$ is an invariant measure $\mu_\ph$ such that $P(\ph) = h_{\mu_\ph}(f) + \int\ph\,d\mu_\ph$. For the potential function $\ph=0$, the equilibrium state $\mu_\ph=\mu_0$ is the \emph{measure of maximal entropy}.

\subsection{First proof of Theorem \ref{thm:es}: symbolic representation of $f|_{\Lambda}$}\label{sec:Markov}
The original proof of Theorem \ref{thm:es} uses a symbolic coding of the dynamics on 
$\Lambda$.  If $\Lambda = X_1 \cup \cdots \cup X_p$, then we say that a bi-infinite sequence $\omega \in \Omega_p := \{1,\dots, p\}^\mathbb{Z}$ \emph{codes the orbit} of $x\in \Lambda$ if $f^n x \in X_{\omega_n}$ for all $n\in \mathbb{Z}$.  When $\Lambda$ is a locally maximal hyperbolic set, it is not hard to show that every $\omega\in \Omega_p$ codes the orbit of at most one $x\in \Lambda$; if such an $x$ exists, call it $\pi(\omega)$.  Let $\Sigma\subset \Omega_p$ be the set of all sequences that code the orbit of some 
$x\in \Lambda$; then $\Sigma$ is invariant under the shift map 
$\sigma\colon \Omega_p\to \Omega_p$ defined by $(\sigma\omega)_n = \omega_{n+1}$, and the map $\pi\colon \Sigma\to \Lambda$ is a \emph{topological semi-conjugacy}, meaning that the following diagram commutes.
\begin{equation}\label{eqn:tsc}
\xymatrix{
\Sigma \ar[r]^\sigma \ar[d]^\pi & \Sigma \ar[d]^\pi \\
\Lambda \ar[r]^f & \Lambda
}
\end{equation}
If the sets $X_i$ overlap, then the coding map may fail to be injective. One would like to produce a coding space $\Sigma$ with a `nice' structure for which the failure of injectivity is `small'. This was accomplished by Sinai when $\Lambda=M$ \cite{yS68} and by Bowen in the general setting \cite{rB70,Bow}; they showed that things can be arranged so that 
$\Sigma$ is defined by a nearest-neighbor condition, with the failure of injectivity confined to sets that are invisible from the point of view of equilibrium states.

\begin{theorem}[\cite{Bow}]\label{thm:Markov}
If $\Lambda$ is a locally maximal hyperbolic set for a diffeomorphism $f$, then there is a \emph{Markov partition} $\Lambda = X_1\cup \cdots \cup X_p$ such that 
each $X_i$ is a rectangle that is the closure of its interior  (in the induced topology on $\Lambda$) and
the corresponding coding space $\Sigma$ is a \emph{topological Markov chain}
\begin{equation}\label{eqn:tmc}
\Sigma = \{ \omega \in \Omega_p : f(\Int X_{\omega_n}) \cap \Int X_{\omega_{n+1}} \neq\emptyset \text{ for all } n\in \mathbb{Z} \},
\end{equation}
and there is a set $\Lambda' \subset \Lambda$ such that 
\begin{enumerate}[label=\textup{(\arabic{*})}]
\item every $x\in \Lambda'$ has a unique preimage under $\pi$, and
\item if $\mu\in \mathcal{M}(f,\Lambda)$ is an equilibrium state for a H\"older continuous potential $\ph\colon \Lambda\to \RR$, then $\mu(\Lambda')=1$.
\end{enumerate}
\end{theorem}
With this result in hand, the problem of existence and uniqueness of equilibrium states can be transferred from the smooth system $(\Lambda,f,\ph)$ to the symbolic system $(\Sigma,\sigma,\tilde\ph := \ph\circ \pi)$, where tools from statistical mechanics and Gibbs distributions can be used; we recall here the most important ideas, referring to \cite{Sin,Bow,Rue} for full details.

Give $\Sigma$ the metric $d(\omega,\omega') = 2^{-\min\{|n|:\omega_n \neq\omega_n'\}}$, so that two sequences are close if they agree on a long interval of integers around the origin.  The coding map $\pi$ is H\"older continuous in this metric, so $\tilde\ph$ is also H\"older continuous. Fixing $r\in (\frac 12,1)$, the Bowen balls associated to the dynamical metric \eqref{eqn:dn} are given by
\begin{equation}\label{eqn:Cn}
B_n(\omega,r) = \{ \omega'\in \Sigma : \omega'_i=\omega_i\text{ for all }0\leq i < n\}
=: C_n(\omega),
\end{equation}
which we call the \emph{$n$-cylinder} of $\omega$.
Let $E_n\subset \Sigma$ contain exactly one point from each $n$-cylinder; then $E_n$ is both $(n,r)$-spanning and $(n,r)$-separated, and writing $Z_n(\Sigma,\tilde\ph) = \sum_{\omega\in E_n} e^{S_n\tilde\ph(\omega)}$, one obtains
\[
P(\ph) = P(\tilde\ph) = \lim_{n\to\infty} \frac 1n \log Z_n(\Sigma,\tilde\ph).
\]
To understand what an equilibrium state for $\tilde\ph$ should look like, recall that
the Kolmogorov--Sinai entropy of a $\sigma$-invariant measure $\tilde\mu$ is defined as
\[
h_{\tilde\mu}(\sigma) = \lim_{n\to\infty} \frac 1n \sum_{\omega\in E_n} -\tilde\mu(C_n(\omega)) \log \tilde\mu(C_n(\omega)).
\]
A short exercise using invariance of $\tilde\mu$ and continuity of $\ph$ shows that
\[
\int\tilde\ph\,d\tilde\mu 
= \lim_{n\to\infty} \sum_{\omega\in E_n} \int_{C_n(\omega)} \tilde\ph\,d\tilde\mu
= \lim_{n\to\infty} \sum_{\omega\in E_n} \tilde\mu(C_n(\omega)) \cdot\frac 1n S_n\tilde\ph(\omega).
\]
Thus maximizing $h_{\tilde\mu}(\sigma) + \int\tilde\ph\,d\tilde\mu$ involves maximizing the limit of a sequence of expressions of the form
$F(p_1,\dots,p_N)=\sum_{i=1}^N p_i (-\log p_i + a_i)$,
where $N(n)=\#E_n$ and $p_i,a_i$ are given by $\tilde\mu(C_n(\omega))$ and $S_n \tilde\ph(\omega)$, so that $p_i\geq 0$ and $\sum_i p_i = 1$.  It is a calculus exercise to show that with $a_i$ fixed, $F$ achieves its maximum value of $F = \log \sum e^{a_j} = \log Z_n(\Sigma,\tilde\ph) \approx nP(\ph)$ when
$p_i = e^{a_i} / \sum_j e^{a_j} \approx e^{a_i} e^{-nP(\ph)}$.

This last relation can be rewritten as 
$\tilde\mu(C_n(\omega))\approx e^{S_n\tilde\ph(\omega)} e^{-nP(\ph)}$. With this in mind, one can use tools from functional analysis and statistical mechanics to show that there is a $\sigma$-invariant ergodic measure $\tilde\mu$ on $\Sigma$ which has the \emph{Gibbs property} with respect to $\tilde\ph$: there is $Q>0$ such that for every $\omega\in\Sigma$ and $n\in\NN$, we have
\begin{equation}\label{gibbs1}
Q^{-1}\leq \frac{\tilde\mu(C_n(\omega))}{\exp(-P(\ph)n+S_n\tilde\ph(\omega))}\leq Q,
\end{equation}
By a general result that we will state momentarily, this is enough to guarantee that 
$\tilde\mu$ is the unique equilibrium state for $(\Sigma,\sigma,\tilde\ph)$, and hence by Theorem \ref{thm:Markov}, its projection $\mu = \pi_* \tilde\mu$ is the unique equilibrium state for $(\Lambda,f,\ph)$.  

To formulate the link between the Gibbs property and equilibrium states, we first recall the following more general definitions.
\begin{definition}\label{def:expansive}
Given a compact metric space $X$, a homeomorphism $f\colon X\to X$ is said to be \emph{expansive} if there is $\eps>0$ such that every $x\neq y\in X$ have $d(f^nx,f^ny)>\eps$ for some $n\in \mathbb{Z}$.  
\end{definition}

\begin{definition}\label{def:Gibbs}
A measure $\mu$ on $X$ is a \emph{Gibbs measure} for $\ph\colon X\to \RR$ if for every small $r>0$ there is $Q=Q(r)>0$ such that for every $x\in X$ and $n\in \NN$, we have
\begin{equation}\label{gibbs2}
Q^{-1}\leq \frac{\mu(B_n(x,r))}{\exp(-P(\varphi)n+S_n\ph(x))}\leq Q.
\end{equation}
\end{definition}
Note that $\sigma\colon \Sigma\to\Sigma$ is expansive, and that \eqref{gibbs1} implies \eqref{gibbs2} in this symbolic setting.  Then uniqueness of the equilibrium state is a consequence of the following general result.

\begin{proposition}[{\cite[Lemma 8]{rB745}}]\label{prop:unique}
If $X$ is a compact metric space, $f\colon X\to X$ is an expansive homeomorphism, and $\mu$ is an ergodic $f$-invariant Gibbs measure for $\ph\colon X\to \RR$, then $\mu$ is the unique equilibrium state for $\ph$.
\end{proposition}

We remark that \eqref{gibbs2} does not require the Gibbs measure to be invariant.  Indeed, one can separate the problem of finding a unique equilibrium state into two parts: first construct a Gibbs measure without worrying about whether or not it is invariant, then find a density function (bounded away from $0$ and $\infty$) that produces an ergodic invariant Gibbs measure, which is the unique equilibrium state by Proposition \ref{prop:unique}.

\subsection{Second proof of Theorem \ref{thm:es}: specification property}\label{gibbs-property}
There is another proof of Theorem \ref{thm:es} which is due to Bowen \cite{rB745} and avoids symbolic dynamics.  Instead, it  uses the fact that $f$ satisfies the following \emph{specification property} on a topologically mixing locally maximal hyperbolic set 
$\Lambda$: for each $\delta>0$ there is an integer $p(\delta)$ such that given any 
points $x_1,\ldots ,x_n\in \Lambda$ and intervals of integers
$I_1,\ldots I_n \subset[a,b]$ with $d(I_i,I_j)\geq p(\delta)$ for $i\neq j$, there is a point 
$x\in \Lambda$ with $f^{b-a+p(\delta)}(x)=x$ and $d(f^k(x),f^k(x_i))<\delta$ for $k\in I_i$.
Roughly speaking, $f$ satisfies specification if for every finite number of orbit segments one can find a single periodic orbit that consecutively approximates each segment with a fixed precision $\delta>0$, and such that transition times are bounded by $p(\delta)$. This property allows one to study some topological and statistical properties of $f$ by only analyzing periodic orbits.

The construction of the Gibbs measure $\tilde\mu$ in the first approach uses eigendata of a certain linear operator acting on an appropriately chosen Banach space of functions on $\Lambda$.  The specification property allows one to use a more elementary construction and obtain a Gibbs measure on $\Lambda$ as a weak* limit point of measures supported on periodic orbits.  Let $\Per_n := \{x\in \Lambda : f^nx=x\}$ and 
$\Zper_n(\ph) := \sum_{x\in \Per_n} e^{S_n\ph(x)}$ (compare this to $\Zspan_n$ and 
$\Zsep_n$ from Definition \ref{def:pressure}); then consider the $f$-invariant Borel probability measures given by
\[
\mu_{n}:= \frac{1}{\Zper_n(\ph)} \sum_{x\in \Per_n}e^{S_n\ph(x)}\delta_x,
\]
where $\delta_x$ is the atomic probability measure with $\delta_x(\{x\})=1$.

Using some counting estimates on the partition sums $\Zper_n(\ph)$ provided by the specification property, one can prove that every weak* limit point $\mu$ of the sequence $\mu_n$ is an ergodic Gibbs measure as in \eqref{gibbs2}.  Then Proposition \ref{prop:unique} shows that $\mu$ is the unique equilibrium state for $\ph$; \emph{a posteriori}, the sequence $\mu_n$ converges.

\section{Description of reference measures and main results}\label{sec:main-result}

In this section, and especially in Theorem \ref{thm:main}, we describe 
%The main result of this paper, which we state as Theorem \ref{thm:main} in \S\ref{sec:statements}, is 
a new proof of Theorem \ref{thm:es} that avoids Markov partitions and the specification property, and instead mimics the geometric construction of SRB measures in \S\ref{sec:physical}.  Given a locally maximal hyperbolic set $\Lambda$ and a H\"older continuous potential $\ph\colon\Lambda\to\RR$, we define for each $x\in\Lambda$ a measure $m_x^\C$ on $X = \Vl^u(x) \cap \Lambda$ such that the sequence of measures
\begin{equation}\label{u-evolution}
\mu_n=\frac1n\sum_{k=0}^{n-1} \frac{f_*^k m_x^\C}{m_x^\C(X)}
\end{equation} 
converges to the unique equilibrium state.\footnote{To be more precise we need first to extend $m_x^\C$ from $X$ to a measure on $\Lambda$ by assigning to any Borel set $E\subset\Lambda$ the value $m_x^\C(E\cap X)$. We shall always assume that in \eqref{u-evolution} $m_x^\C$ is extended in this way.}  In \S\ref{sec:ref}, we give some motivation for the properties we require the reference measures $m_x^\C$ to have; then in \S\ref{sec:car} we explain our construction of these measures.
In \S\ref{sec:statements} we state our main results establishing the properties of $m_x^\C$, including how these measures can be used to prove Theorem \ref{thm:es}.  %These results are proved in \S\S\ref{sec:measures}--\ref{sec:geometric}.
In \S\ref{sec:proofs} we outline the proofs of these results, referring to \cite{CPZ} for complete details and for proofs of various technical lemmas.

\subsection{Conditional measures as reference measures}
We start with the observation that if we were already in possession of the equilibrium state $\mu$, then the conditional measures of $\mu$ would immediately define reference measures for which the construction just described produces $\mu$.  Indeed, suppose $X$ is a compact topological space, $f\colon X\to X$ a continuous map, and 
$\mu$ a finite $f$-invariant ergodic Borel probability measure on $X$. Given $Y\subset X$ with $\mu(Y)>0$ and a measurable partition $\xi$ of $Y$; let $\tilde\mu$ be the corresponding factor-measure on $\tilde{Y}:=Y/\xi$, and $\{ \mu_W^\xi : W\in \xi\}$ the conditional measures on partition elements.\footnote{Note that $\xi$ is not assumed to have any dynamical significance; in particular it need not be a partition into local unstable leaves, although this is the most relevant partition for our purposes.}  We prove the following result in \S\ref{sec:push-cond}; it follows from an even more general result in ergodic theory that we state below as Proposition \ref{prop:generic-pts}.

\begin{theorem}\label{push-cond1}
For $\tilde\mu$-almost every $W\in \xi$, any probability measure $\nu$ on $W$ such that $\nu \ll \mu_W^\xi$ has the property that $\nu_n:= \frac{1}{n}\sum_{k=0}^{n-1}f^k_{*}\nu$ converges in the weak$^*$ topology to the measure $\mu$.
\end{theorem}

Of course, Theorem \ref{push-cond1} is not much help in finding the equilibrium state $\mu$, because we need to know $\mu$ to obtain the conditional measures $\mu_W^\xi$.  We must construct the reference measure $m_x^\C$ independently, without using any knowledge of existence of equilibrium states.  Once we have done this, we will eventually show that $m_x^\C$ is equivalent to the conditional measure of the constructed equilibrium state, so our approach not only allows us to develop a new way of constructing equilibrium states, but also describes their conditional measures. 

\subsection{Conditions to be satisfied by reference measures}
\label{sec:ref}
To motivate the properties that our reference measures must have, we first consider the specific case when $\L$ is an attractor and outline the steps in constructing SRB measures.
\begin{enumerate}
\item Given a local unstable leaf $W=\Vl^u(x)$ through $x\in\L$ and $n\ge 0$, the image $W_n:=f^n(W)$ is contained in the union of local leaves $\Vl^u(y_i)$ for some points $y_1,\dots,y_s\in W_n$, and leaf volume $m_x^u$ is pushed forward to a measure $f_*^n m_x^u$ such that $(f_*^n m_x^u)|_{\Vl^u(y_i)} \ll m_{y_i}^u$ for each $i$.
\item Each $\mu_n:=\frac1n\sum_{k=0}^{n-1}f_*^k m_x^u$ can be written as a convex combination of measures with the form $\rho_n^i \,dm_{y_i}^u$ for some  functions $\rho_n^i \colon\Vl^u(y_i)\to [0,\infty)$ that are uniformly bounded away from $0$ and $\infty$.
\item To show that any limit measure $\mu = \lim_{j\to\infty} \mu_{n_j}$ has absolutely continuous conditional measures on unstable leaves, first observe that given a rectangle $R$, the partition $\xi$ into local unstable leaves can be approximated by a refining sequence of finite partitions $\xi_\ell$, and the conditional measures $\mu_x^\xi$ are the weak* limits  of the conditional measures $\mu_x^{\xi_\ell}$ as $\ell\to\infty$.
\item 
The bounds on the density functions $\rho_n^i$ allow us to control the conditional measures $\mu_x^{\xi_\ell}$, and hence to control $\mu_x^\xi$ as well; in particular, these measures  are absolutely continuous with respect to leaf volume, and thus $\mu$ is an SRB measure.
\end{enumerate}

Now we describe two crucial properties of the leaf volumes $m_x^u$, which we will eventually need to mimic with our reference measures $m_x^\C$.  The first of these already appeared in \eqref{eqn:volume-scaling}, and describes how $m_x^u$ scales under iteration by $f$; this will let us conclude that the SRB measure $\mu$ is an equilibrium state for $\phigeo$.  The second property describes how $m_x^u$ behaves when we `slide along stable leaves' via a \emph{holonomy map}; this issue has so far been ignored in our discussion, but plays a key role in the proof that the SRB measure $\mu$ is ergodic, and hence is the \emph{unique} equilibrium state for $\phigeo$.

\subsubsection{Scaling under iteration}
Given any $x,y\in \Lambda$ and $A\subset f(\Vl^u(x)) \cap \Vl^u(y)$, we have
\[
f_* m_x^u(A) = m_x^u(f^{-1}A) = \int_{A} |\det f^{-1}|_{E^u(z)}| \,dm_y^u(z)
= \int_{A} e^{\phigeo(f^{-1}z)}\,dm_y^u(z)
\]
and so the Radon--Nikodym derivative comparing the family of measures $m_x^u$ to their pushforwards is given in terms of the geometric potential:
\begin{equation}\label{eqn:rn}
\frac{d(f_* m_x^u)}{dm_{y}^u}(z) = e^{\phigeo(f^{-1}(z))}.
\end{equation}
Iterating this, we see that given $A\subset f^n(\Vl^u(x)) \cap \Vl^u(y)$ we have
\begin{equation}\label{eqn:push-n}
f_*^n m_x^u(A)=m_x^u(f^{-n}A) =  \int_{A} e^{\sum_{k=1}^n \phigeo(f^{-k}z)}\,dm_{y}^u(z).
\end{equation}
By H\"older continuity of $\phigeo$ and the fact that $f^{-1}$ contracts uniformly along each $\Vl^u$, one can easily show that
\begin{equation}\label{eqn:bdd-dist}
\Big| \sum_{k=1}^n \big(\phigeo(f^{-k} z_1) - \phigeo(f^{-k}z_2)\big)\Big| \leq Q_u \text{ for all } z_1,z_2 \in \Vl^u(y),
\end{equation}
where $Q_u$ is a constant independent of $y \in \Lambda$, $z_1,z_2\in \Vl^u(y)$, and 
$n\in \NN$ (see Lemma \ref{lem:bowen} for details).  Together with \eqref{eqn:push-n}, this gives
\begin{equation}\label{eqn:bdd-dist-2}
e^{-Q_u} \leq \frac{m_x^u(f^{-n} A)}{e^{S_n\phigeo(x)} m_{y}^u(A)} \leq e^{Q_u}
\text{ for all } A\subset f^n(\Vl^u(x)) \cap \Vl^u(y).
\end{equation}
In particular, writing $B^u(y,r) = B(y,r) \cap \Vl^u(y)$, we observe that for each $r>0$ there is a constant $K=K(r)>0$ such that $m_y^u(B^u(y,r)) \in [K^{-1},K]$ for all $y\in \Lambda$, and deduce from \eqref{eqn:bdd-dist-2} that the \emph{$u$-Bowen ball}
\[
B_n^u(x,r) := \{z\in \Vl^u(x) : d_n(x,z) < r\} = f^{-n} B^u(f^n x,r)
\]
admits the following leaf volume estimate:
% for $Q_1 = K(r) e^{Q_u}$:
%some constant $Q_1=Q_1(r)$, independent of both $x\in \Lambda$ and $n\in \NN$:
\begin{equation}\label{eqn:mW-gibbs}
K^{-1} e^{-Q_u} \le\frac{m^u_x(B_n^u(x,r))}{e^{S_n\phigeo(x)}}\le K e^{Q_u}.
\end{equation}

\begin{definition}\label{def:u-gibbs}
Consider a family of measures $\{\mu_x : x\in \Lambda\}$ such that $\mu_x$ is supported on $\Vl^u(x)$. We say that this family has the \emph{$u$-Gibbs property}\footnote{Note that this is a different notion than the idea of \emph{$u$-Gibbs state} from \cite{PS82}.} with respect to the potential function $\ph\colon \Lambda\to \RR$ if there is $Q_1=Q_1(r)>0$ such that for all $x\in \Lambda$ and $n\in \NN$, we have
\begin{equation}\label{eqn:u-gibbs}
Q_1^{-1} \leq \frac{\mu_x(B_n^u(x,r))}{e^{-nP(\ph) + S_n\ph(x)}} \leq Q_1.
\end{equation}
\end{definition}
In particular, \eqref{eqn:mW-gibbs} says that $m_x^u$ has the \emph{$u$-Gibbs property} with respect to the potential function $\phigeo$. Since the SRB measure $\mu$ constructed above has conditional measures that are given by multiplying the leaf volumes $m_x^u$ by `nice' density functions, one can use \eqref{eqn:mW-gibbs} to ensure that
the conditional measures of $\mu$ also have the $u$-Gibbs property; integrating these conditional measures gives the Gibbs property for $\mu$, and then some straightforward estimates involving $\Zspan_n(\L,\ph,r)$ demonstrate that $\mu$ is an equilibrium state corresponding to the function $\phigeo$.

\subsubsection{Sliding along stable leaves}\label{sliding}

It remains, then, to show that $\mu$ is the \emph{unique} equilibrium state for $\phigeo$; this will follow from Proposition \ref{prop:unique} if $\mu$ is proved to be ergodic.  To establish ergodicity we use the \emph{Hopf argument}, which goes back to E.\ Hopf's work on geodesic flow over surfaces \cite{eH39}.  The first step is to observe that if $\mu$ is any invariant measure, then by Birkhoff's ergodic theorem,\footnote{This is a more general version of the ergodic theorem than the one we mentioned in \S\ref{sec:hta}; this version applies even when 
$\mu$ is not ergodic, but does not require that the limits in \eqref{eqn:BET} are equal to $\int\psi\,d\mu$; instead, one obtains $\int \overline\psi\,d\mu = \int\psi\,d\mu$, which implies the earlier version in the case when $\overline\psi$ is constant $\mu$-a.e.} for every $\psi\in L^1(\mu)$, the forward and backward ergodic averages exist and agree for $\mu$-a.e.\ $x$:
\begin{equation}\label{eqn:BET}
\lim_{n\to\infty} \frac 1n \sum_{k=0}^{n-1} \psi(f^k x) =
\lim_{n\to\infty} \frac 1n \sum_{k=0}^{n-1} \psi(f^{-k} x).
\end{equation}
Let $\mathcal{B}\subset \L$ be the set of points where the limits in \eqref{eqn:BET} exist and agree for every continuous $\psi\colon \L\to \RR$; such points are called \emph{Birkhoff regular}.  For each $x\in \mathcal{B}$, write $\bpsi(x)$ for the common value of these limits; note that $\bpsi$ is defined $\mu$-a.e. It is not hard to prove that $\mu$ is ergodic if and only if the function $\bpsi\colon \mathcal{B}\to\RR$ is constant $\mu$-a.e.\ for every continuous $\psi\colon \L\to \RR$. By topological transitivity and the fact that $\bpsi \circ f=\bpsi$ on $\mathcal{B}$, one obtains the following standard result, whose proof we omit.

\begin{lemma}\label{lem:ergodic-condition}
An $f$-invariant measure $\mu$ is ergodic if and only if for every continuous $\psi\colon \L\to\RR$ and every rectangle $R\subset \L$, the function $\bpsi\colon \mathcal{B}\cap R \to\RR$ is constant $\mu$-a.e.
\end{lemma}

Now comes the central idea of the Hopf argument: given $\psi\in C(\L)$, if $\bpsi(x)$ exists then a short argument using the left-hand side of \eqref{eqn:BET} gives $\bpsi(y)=\bpsi(x)$ for all $y\in \mathcal{B} \cap V_R^s(x)$.  Similarly, $\bpsi$ is constant on $\mathcal{B}\cap V_R^u(x)$ using the right-hand side of \eqref{eqn:BET}.  

We want to conclude the proof of ergodicity by saying something like the following: \emph{``Since $\mathcal{B}$ has full measure in $R$, it has full measure in almost every stable and unstable leaf in $R$; thus there is $p\in R$ such that $\mathcal{B}_p:=\bigcup_{x\in \mathcal{B}\cap V_R^u(p)}\mathcal{B}\cap V_R^s(x)$ has full measure in $R$, and by the previous paragraph, $\bpsi$ is constant on $\mathcal{B}_p$, so Lemma \ref{lem:ergodic-condition} applies.''}

There is a subtlety involved in making this step rigorous.  To begin with, the term ``full measure'' is used in two different ways: ``$\mathcal{B}$ has full measure in $R$'' means that its complement $\mathcal{B}^c = R\setminus \mathcal{B}$ has $\mu(R\setminus \mathcal{B})=0$, while ``$\mathcal{B}$ has full measure in the stable leaf $V_R^s(x)$'' means that $\mu_x^s(\mathcal{B}^c)=0$, where $\mu_x^s$ is the conditional measure of $\mu$ along the stable leaf.  Using the analogue of \eqref{gibbs-split}--\eqref{gibbs-split2} for the decomposition into stable leaves, we have
\begin{equation}\label{eqn:conditional2}
\begin{aligned}
\mu(\mathcal{B}^c) = \int_R \mu_x^s(\mathcal{B}^c) \,d\mu(x)
= \int_{V_R^u(p)} \mu_x^s(\mathcal{B}^c) \,d\tilde\mu_p(x),
\end{aligned}
\end{equation}
where $\tilde\mu_p$ is the measure on $V_R^u(p)$ defined by 
\[
\tilde\mu_p(A) = \mu\big(\bigcup_{x\in A} V_R^s(x)\big).
\]
Let $\mathcal{B}' = \{x\in R : \mu_x^s(\mathcal{B}^c)=0\}$.  It follows that
\[
0 = \mu(\mathcal{B}^c) = \int_R \mu_x^s(\mathcal{B}^c)\,d\mu(x)
= \int_{R\setminus \mathcal{B}_0} \mu_x^s(\mathcal{B}^c)\,d\mu(x),
\]
and since $\mu_x^s(\mathcal{B}^c)>0$ for all $x\in R\setminus \mathcal{B}_0$ by definition, we conclude that $\mu(R\setminus \mathcal{B}')=0$; in other words, $\mu_x^s(\mathcal{B}^c)=0$ for $\mu$-a.e.\ $x\in R$.
A similar argument produces $\mathcal{B}'' \subset \mathcal{B}'$ such that $\mu(R\setminus \mathcal{B}'') = 0$ and $\mu_x^u(\mathcal{B}^c)=\mu_x^s(\mathcal{B}^c)=0$ for every $x\in \mathcal{B}''$.

So far, things are behaving as we expect.
Now can we conclude that $\mu(\mathcal{B}_p) = \mu(R)$ for $p\in \mathcal{B}''$, thus completing the proof of ergodicity?  Using \eqref{eqn:conditional2}, we have
\[
\mu(\mathcal{B}_p) = \int_{V_R^u(p)} \mu_x^s\bigg( \bigcup_{y\in \mathcal{B} \cap V_R^u(p)} \mathcal{B}\cap V_R^s(y)\bigg)
 \,d\tilde\mu_p(x)
\geq \int_{\mathcal{B}'' \cap V_R^u(p)} \mu_x^s(\mathcal{B}) \,d\tilde\mu_p(x)
= \tilde\mu_p(\mathcal{B}'').
\]
We would like to say that $\tilde\mu_p(\mathcal{B}'') = \tilde\mu_p(V_R^u(p)) = \mu(R)$, and conclude that $\mathcal{B}_p$ has full $\mu$-measure in $R$.  We know that $\mu_p^u(\mathcal{B}'') = \mu_p^u(V_R^u(p))$, and so the proof will be complete if the answer to the following question is ``yes''.

\begin{question*}
Are the measures $\tilde\mu_p$ and $\mu_p^u$ on $V_R^u(p)$ equivalent?
\end{question*}

Note that the measures $\mu_p^u$ are defined in terms of the foliation $V_R^u$, while the measures $\tilde\mu_p$ are defined in terms of the foliation $V_R^s$.  We can write the measures $\tilde\mu_p$ in terms of $\mu_p^u$ as follows: given $A\subset V_R^u(p)$, we have
\begin{equation}\label{eqn:mu-tilde-mu}
\tilde\mu_p(A) = \mu\bigg( \bigcup_{x\in A} V_R^s(x)\bigg)= \int_R \mu_y^u 
\{ V_R^u(y) \cap V_R^s(x) : x\in A\}\,d\mu(y).
\end{equation}
For each $p,y\in R$, consider the \emph{(stable) holonomy map} $\pi_{py} \colon V_R^u(p) \to V_R^u(y)$ defined by $\pi_{py}(x) = V_R^u(y) \cap V_R^s(x)$, which maps one unstable leaf to another by \emph{sliding along stable leaves}; see Figure \ref{fig:holonomy}.  Then \eqref{eqn:mu-tilde-mu} becomes
\[
\tilde \mu_p(A) = \int_R (\mu_y^u\circ \pi_{py})(A) \,d\mu(y).
\]
In other words, $\tilde\mu_p$ is the average of the conditional measures $\pi_{py}^* \mu_y^u = \mu_y^u \circ \pi_{py}$ taken over all $y\in R$.

\begin{figure}[htbp]
\includegraphics[width=.35\textwidth]{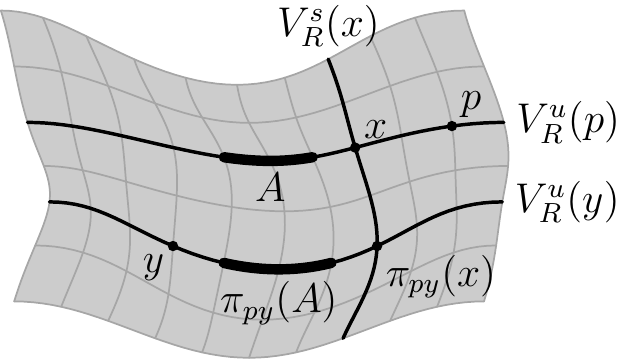}
\caption{The stable holonomy map from $V_R^u(p)$ to $V_R^u(y)$.}
\label{fig:holonomy}
\end{figure}

\begin{definition}\label{def:abs-cts}
Let $\{\nu_x\}_{x\in R}$ be a family of measures on $R$ with the property that each $\nu_x$ is supported on $V_R^u(x)$, and $\nu_x = \nu_y$ whenever $y\in V_R^u(x)$.  We say that the family $\{\nu_x\}$ is \emph{absolutely continuous\footnote{There is a related, but distinct, notion of absolute continuity of a foliation (with respect to volume), which also plays a key role in smooth ergodic theory; see \cite[\S8.6]{pes-red}.}
with respect to stable holonomies} if $\pi_{xy}^* \nu_y \ll \nu_x$ for all $x,y\in R$.
\end{definition}

The preceding arguments lead to the following result; full proofs (and further discussion) can be found in \cite{CHT16}.

\begin{proposition}\label{prop:hopf}
Let $\L$ be a topologically transitive hyperbolic set for a $C^{1+\alpha}$ diffeomorphism $f$, and let $\mu$ be an $f$-invariant measure on $\L$.  Suppose that for every rectangle $R\subset \Lambda$ with $\mu(R)>0$, the unstable conditional measures $\mu_x^u$ are absolutely continuous with respect to stable holonomies.  Then $\mu$ is ergodic.
\end{proposition}

\subsection{Construction of reference measures}\label{sec:car} 

In light of the previous section, our goal is to construct for each potential $\ph$ a reference measure $m_x^\C$ on each leaf $\Vl^u(x)$ satisfying a property analogous to \eqref{eqn:rn} with $\ph$ in place of $\phigeo$, together with the absolute continuity property from Definition \ref{def:abs-cts}.

From now on we fix a local unstable manifold $W = \Vl^u(x)$ of size $\tau$ and consider the set $X = W\cap \Lambda$ on which we will build our reference measure. Before treating general potentials, we start with the geometric potential $\phigeo$, and we assume that $\Lambda$ is an attractor for $f$, so that  $W\subset\Lambda$. This is necessary for the moment since the measure we build will be supported on $W \cap \Lambda$, and the support of $m_W$ is all of $W$; in the general construction below we will not require $\Lambda$ to be an attractor. For the geometric potential $\phigeo$, we know the reference measure $m_x^\C$ should be equivalent to leaf volume $m_W$ on $W$.\footnote{From Theorem \ref{push-cond1} we see that the equivalence class of the measure is the crucial thing for the geometric construction to work.} Leaf volume is equivalent to the Hausdorff measure $m_H(\cdot,\alpha)$ with $\alpha=\dim E^u$, which is defined by
\begin{equation}\label{eqn:malpha0}
m_H(Z,\alpha):=\lim_{\eps\to 0}\inf\sum_{i=1}^\infty (\diam U_i)^\alpha,
\end{equation}
where the infimum is taken over all collections $\{U_i\}$ of open sets $U_i\subset W$ with 
$\diam U_i\leq \eps$ which cover $Z\subset W$.  

We want to describe a measure that is equivalent to $m_H(\cdot,\alpha)$ but whose definition uses the dynamics of $f$. In \eqref{eqn:malpha0}, the covers used to measure $Z$ were refined geometrically by sending $\eps\to 0$.  We consider instead covers that refine \emph{dynamically}: we restrict the sets $U_i$ to be $u$-Bowen balls 
$B_n^u(x,r)=B_n(x,r)\cap W$, and refine the covers by requiring $n$ to be large rather than by requiring $r$ to be small. Note that if $U_i$ is a metric ball $B(x,\eps)$, then 
$(\diam U_i)^\alpha \approx m_W(U_i)$ up to a multiplicative factor that is bounded away from $0$ and $\infty$.  For a $u$-Bowen ball, on the other hand, \eqref{eqn:mW-gibbs} gives $m_W(B_n^u(x,r)) \approx e^{S_n\phigeo(x)}$, and so we use this quantity to compute the weight of the cover. This suggests that we should fix $r>0$ and define the measure of $Z\subset W$ by 
\begin{equation}\label{vol.car}
m_x^{\phigeo}(Z)
:=\lim_{N\to\infty}\inf\sum_{i=1}^{\infty} e^{S_{n_i}\phigeo(x_i)},
\end{equation}
where the infimum is taken over all collections $\{B^u_{n_i}(x_i,r)\}$ of $u$-Bowen balls with $x_i\in W$, $n_i\ge N$, which cover $Z$. It is relatively straightforward to derive property (\ref{eqn:rn}) from (\ref{vol.car}).

Now it is nearly apparent what the definition should be for a general potential; we want to replace $\phigeo$ with $\ph$ in \eqref{vol.car}. There is one small subtlety, though. First, Definition \ref{def:pressure} gives $P(\ph+c) = P(\ph)+c$ for $c\in \RR$. This along with the definition of equilibrium state and the variational principle \eqref{eqn:vp2} shows that adding a constant to $\ph$ does not change its equilibrium states, and thus we should also expect that $\ph$ and $\ph+c$ produce the same reference measure on $W\cap \Lambda$. For this to happen, we need to modify \eqref{vol.car} so that adding a constant to $\ph$ does not affect the value. This can be achieved by multiplying each term in the sum by $e^{-n_iP(\ph)}$; note that since $P(\phigeo)=0$ this does not modify \eqref{vol.car}. Thus we make the following definition.

\begin{definition}\label{def:main}
Let $0<r<\frac{\tau}{3}$. We define a measure on $X :=W\cap \Lambda$ by
\begin{equation}\label{car0}
m_x^\C(Z):=\lim_{N\to\infty}\inf \sum_i e^{-n_iP(\ph)} e^{S_{n_i}\varphi(x_i)} ,
\end{equation}
where the infimum is taken over all collections $\{B^u_{n_i}(x_i,r)\}$ of $u$-Bowen balls with $x_i\in X$, $n_i\ge N$, which cover $Z$, and for convenience we write 
$\C=(\ph,r)$ to keep track of the data on which the reference measure depends.
\end{definition}
Both definitions \eqref{vol.car} and \eqref{car0} are specific cases of the Carath\'eodory measure produced by a dynamically defined Carath\'eodory dimension structure, which we discuss at greater length in \S\ref{car-struc}; this is the Pesin--Pitskel' definition of topological pressure \cite{PP84} that generalized Bowen's definition of topological entropy for non-compact sets \cite{rB73}.  In particular, Proposition \ref{prop:P-same} establishes the crucial property that every local unstable leaf carries the same topological pressure as the entire set $\L$.

\subsection{Statements of main results}\label{sec:statements}

Now we state the most important properties of $m_x^\C$ and show how it can be used as a reference measure to construct the equilibrium state for $\ph$.
All results in this section are proved in detail in \cite{CPZ};\footnote{The numbering of references within \cite{CPZ} refers to the first arXiv version; it is possible that the numbering will change between this and the final published version.}
%see Remarks 2.3 and 4.1 in that paper for an explanation of why it covers our setting here.  
we outline the proofs in \S\ref{sec:proofs}.
Our first main result shows that the measure $m_x^\C$ is finite and nonzero.  

\begin{theorem}\label{thm:finite}{\cite[Theorem 4.2]{CPZ}}
Let $\Lambda$ be a topologically transitive locally maximal hyperbolic set for a $C^{1+\alpha}$ diffeomorphism $f$, and let $\ph\colon \Lambda\to \RR$ be H\"older continuous.  Fix $r$ as in Definition \ref{def:main}, and for each $x\in \Lambda$, let $m_x^\C$ be given by \eqref{car0}, where $\C=(\ph,r)$.  Then there is $K>0$ such that for every $x\in \Lambda$, $m_x^\C$ is a Borel measure on $\Vl^u(x) \cap \Lambda$ with $m_x^\C(\Vl^u(x)\cap \Lambda) \in [K^{-1},K]$.  If $\Vl^u(x) \cap \Vl^u(y) \cap \Lambda \neq\emptyset$, then $m_x^\C$ and $m_y^\C$ agree on the intersection.
\end{theorem}

As described in \S\ref{sec:ref}, we need to understand how the measures $m_x^\C$ transform under (1) the dynamics of $f$ and (2) sliding along stable leaves via holonomy. For the first of these properties, the following result gives the necessary scaling property analogous to \eqref{eqn:rn}.
\begin{theorem}\label{thm:Gibbs}{\cite[Theorem 4.4]{CPZ}}
Under the hypotheses of Theorem \ref{thm:finite}, for every $x\in\Lambda$, we have 
$f^* m_{f(x)}^\C := m_{f(x)}^\C \circ f \ll m_x^\C$, with Radon--Nikodym derivative 
$e^{P(\ph) - \ph}$, so that \eqref{eqn:scaling} holds.  
\end{theorem}
%This result quickly leads to the $u$-Gibbs property from Definition \ref{def:u-gibbs}; 
%We prove the following in \S\ref{sec:gibbs}.

\begin{corollary}\label{cor:Gibbs}{\cite[Corollary 4.5]{CPZ}}
Under the hypotheses of Theorem \ref{thm:finite}, the family of measures 
$\{m_x^\C\}_{x\in \Lambda}$ has the $u$-Gibbs property from Definition \ref{def:u-gibbs}. In particular, for every relatively open $U\subset \Vl^u(x) \cap \L$, we have $m_x^\C(U)>0$.
\end{corollary}

The final crucial property of the reference measures is that they are absolutely continuous under holonomy.

\begin{theorem}\label{thm:holonomy}{\cite[Theorem 4.6]{CPZ}}
Under the hypotheses of Theorem \ref{thm:finite}, there is a constant $Q_2>0$ such that for every rectangle $R\subset \Lambda$ and every $y,z\in R$, the measures 
$\pi_{yz}^*m_z^\C=m_z^\C\circ\pi_{yz}$ and $m_y^\C$ are equivalent on $V_R^u(y)$, with 
\[
Q_2^{-1}\le\frac{d\pi_{yz}^*m_z^\C}{dm_y^\C}\le Q_2.
\]
\end{theorem}
Note that Theorem \ref{thm:holonomy} in particular shows that given a rectangle 
$R\subset \Lambda$, if $m_x^\C(V_R^u(x))>0$ for some $x\in R$, then the same is true for every $x\in R$; moreover, by Corollary \ref{cor:Gibbs} this happens whenever $R$ is the closure of its interior (relative to $\Lambda$).

Using these properties of the measures $m_x^\C$, we can carry out the geometric construction of equilibrium states; see \S\ref{sec:proofs} for the proof of the following.

\begin{theorem}\label{thm:main}
Under the hypotheses of Theorem \ref{thm:finite}, the following are true.
\begin{enumerate}[label=\upshape{(\arabic{*})}]
\item\label{1} For every $x\in \Lambda$, the sequence of measures 
$\mu_n := \frac 1n \sum_{k=0}^{n-1} f_*^k m_x^\C / m_x^\C(\Vl^u(x))$ from \eqref{u-evolution} is weak* convergent as $n\to\infty$ to a probability measure $\mu_\ph$ that is independent of $x$.
\item\label{2} The measure $\mu_\ph$ is ergodic, gives positive weight to every open set in $\L$, has the Gibbs property \eqref{gibbs2} and is the unique equilibrium state for 
$(\Lambda,f,\ph)$.
\item\label{3}
For every rectangle $R\subset \Lambda$ with $\mu_\ph(R)>0$, the conditional measures 
$\mu_y^u$ generated by $\mu_\ph$ on unstable sets $V_R^u(y)$ are equivalent for 
$\mu_\ph$-almost every $y\in R$ to the reference measures $m_y^\C|_{V_R^u(y)}$.  Moreover, there exists $Q_3>0$, independent of $R$ and $y$, such that for 
$\mu_\ph$-almost every $y\in R$ we have\footnote{It is reasonable to expect, based on analogy with the case of SRB measure, that the Radon--Nikodym derivative in \eqref{eqn:mu-m} is in fact H\"older continuous and given by an explicit formula; at present we can only prove this for a modified version of $m_x^\C$, whose definition we omit here.}
\begin{equation}\label{eqn:mu-m}
Q_3^{-1}\leq \frac{d\mu_y^u}{dm_y^\C}(z) m_y^\C(R) \leq Q_3 \text{ for $\mu_y^u$-a.e.}\ z\in V_R^u(y).
\end{equation}
\end{enumerate}
\end{theorem}

Theorems \ref{thm:holonomy} and \ref{thm:main}\ref{3} allow us to show that the equilibrium state $\mu_\ph$ has local product structure, as follows.  Consider a rectangle $R\subset \Lambda$ with $\mu_\ph(R)>0$, and a system of conditional measures $\mu^u_x$ with respect to the partition $\xi$ of $R$ into local unstable leaves.
Given $p\in R$, define a measure $\tilde\mu_p$ on $V_R^s(p)$ by 
$\tilde\mu_p(A)=\mu_\ph(\bigcup_{x\in A}V_R^u(x))$ as in the paragraph preceding \eqref{gibbs-split2}.
%note that by absolute continuity and the discussion preceding Definition \ref{def:abs-cts}, this is equivalent (for $\mu$-a.e.\ $p$) to the conditional measure on the stable leaf.
Since $R$ is homeomorphic to the direct product of $V_R^u(p)$ and $V_R^s(p)$, the product of the measures $\mu_p^u$ and $\tilde\mu_p$ gives a measure on $R$ that we denote by $\mu_p^u\otimes \tilde\mu_p$. The following local product structure result is a consequence of Theorem \ref{thm:holonomy}, Theorem \ref{thm:main}\ref{3}, and \eqref{gibbs-split2}; see \S\ref{sec:lps}.

\begin{corollary}\label{cor:prod}
For every rectangle $R$ and $\mu_\ph$-almost every $p\in R$, we have 
$\pi^*_{py}\mu^u_y \sim \mu^u_p$ for $\tilde\mu_p-$almost every $y\in V_R^s(p)$, and thus $\mu_{\varphi} \sim \mu^u_p\otimes \tilde\mu_p$.
Moreover, it follows that $\tilde\mu_p$ is equivalent to $\mu_p^s$, the conditional measure on $V_R^s(p)$ with respect to the partition into stable leaves, for $\mu_\ph$-a.e.\ $p\in R$.
\end{corollary}

We remark that Corollary \ref{cor:prod} was also proved by Leplaideur \cite{Lep}. His proof uses Markov partitions to construct families of leaf measures with the properties given in Theorems \ref{thm:finite} and \ref{thm:holonomy}.
Historically, this description of $\mu_\ph$ in terms of its direct product structure dates back to Margulis \cite{gM70}, who described the unique measure of maximal entropy for a transitive Anosov flow as a direct product of leafwise measures satisfying the continuous-time analogue of \eqref{eqn:scaling} for $\ph=0$.  In this specific case the equivalences in Corollary \ref{cor:prod} can be strengthened to equalities.

\section{Carath\'eodory dimension structure}\label{car-struc} 

The definition of the measures $m_x^\C$ in \eqref{car0} is a specific instance of the \emph{Carath\'eodory dimension construction} introduced by the second author in \cite{yP88} (see also \cite[\S10]{pes97}). It is a substantial generalization and adaptation to dynamical systems of the classical construction of Carath\'eodory measure in geometric measure theory, of which Lebesgue measure and Hausdorff measure are the most well-known examples.  We briefly recall here the Carath\'eodory dimension construction together with some of its basic properties.

\subsection{Carath\'eodory dimension and measure} 

A \emph{Carath\'eodory dimension structure}, or \emph{$C$-structure}, on a set $X$ is given by the following data.
\begin{enumerate}[label=(\arabic{*})]
\item An indexed collection of subsets of $X$, denoted $\mathcal{F}=\{U_s : s\in \mathcal{S}\}$.
\item Functions $\xi,\eta,\psi\colon \mathcal{S}\to[0,\infty)$ satisfying the following conditions:
\begin{enumerate}[label=\textbf{A\arabic{*}.}]
\item\label{A1} if $U_s=\emptyset$ then $\eta(s)=\psi(s)=0$; if $U_s\neq\emptyset$, then 
$\eta(s)>0$ and $\psi(s)>0$;\footnote{In \cite{pes97}, condition \ref{A1} includes the requirement that there is $s_0\in \mathcal{S}$ such that $U_{s_0} = \emptyset$; %but this can safely be omitted since we can always formally enlarge our collection by adding the empty set, without changing any of the definitions below.
here we remove this assumption and instead define $m_C(\emptyset,\alpha) :=0$, which is equivalent.
}
\item\label{A2} for any $\delta>0$ one can find $\varepsilon>0$ such that 
$\eta(s)\leq\delta$ for any $s\in\mathcal{S}$ with $\psi(s)\leq\varepsilon$;
\item\label{A3} for any $\eps>0$ there exists a finite or countable subcollection $\mathcal{G}\subset \mathcal{S}$ that covers $X$ (meaning that $\bigcup_{s\in \mathcal{G}} U_s \supset X$) and has $\psi(\mathcal{G}) := \sup\{\psi(s) : s\in \mathcal{S} \}\leq \eps$.
\end{enumerate}
\end{enumerate}
Note that no conditions are placed on $\xi(s)$, which we interpret as the \emph{weight} of $U_s$.  The values $\eta(s)$ and $\psi(s)$ can each be interpreted as a \emph{size} or \emph{scale} of $U_s$; we allow these functions to be different from each other.

The $C$-structure $(\mathcal{S},\mathcal{F},\xi,\eta,\psi)$ determines a one-parameter family of outer measures on $X$ as follows.  Fix a nonempty set $Z\subset X$ and consider some $\mathcal{G}\subset \mathcal{S}$ that covers $Z$ (meaning that $\bigcup_{s\in \mathcal{G}} U_s \supset Z$).  
Then $\psi(\mathcal{G})$ is interpreted as the largest size of sets in the cover, and we set for each $\alpha\in \RR$,
\begin{equation}\label{eqn:mCZa}
m_C(Z,\alpha) := \lim_{\eps\to 0} \inf_{\mathcal{G}} \sum_{s\in \mathcal{G}} \xi(s) \eta(s)^\alpha,
\end{equation}
where the infimum is taken over all finite or countable $\mathcal{G}\subset \mathcal{S}$ covering $Z$ with $\psi(\mathcal{G})\le\eps$.  Defining $m_C(\emptyset,\alpha) := 0$, it follows from \cite[Proposition 1.1]{pes97} that $m_C(\cdot,\alpha)$ is an outer measure.
%, this gives an outer measure as long as $m_C(\emptyset,\alpha)=0$, which is automatic for $\alpha>0$ and will be true for all $\alpha\in \RR$ in our applications below.  
The measure induced by $m_C(\cdot,\alpha)$ on the $\sigma$-algebra of measurable sets is the \emph{$\alpha$-Carath\'eodory measure}; it need not be $\sigma$-finite or non-trivial.

\begin{proposition}[{\cite[Proposition 1.2]{pes97}}]\label{prop:C-dim}
For any set $Z\subset X$ there exists a critical value
$\alpha_C\in\mathbb{R}$ such that $m_C(Z,\alpha)=\infty$ for $\alpha<\alpha_C$ and $m_C(Z,\alpha)=0$ for $\alpha>\alpha_C$.
\end{proposition}
We call $\dim_CZ =\alpha_C$ the \emph{Carath\'eodory dimension} of the set $Z$ associated to the $C$-structure $(\mathcal{S},\mathcal{F},\xi,\eta,\psi)$.  By Proposition \ref{prop:C-dim}, $\alpha=\dim_C X$ is the only value of $\alpha$ for which \eqref{eqn:mCZa} can possibly produce a non-zero finite measure on $X$, though it is still possible that $m_C(X,\dim_C X)$ is equal to $0$ or $\infty$.
  
\subsection{Examples of $C$-structures} 
The $C$-structures in which we are interested are generated by other structures on the set $X$.  

\subsubsection{Hausdorff dimension and measure} 
If $X$ is a metric space, then consider the $C$-structure given by
$\mathcal{S} := X\times (0,\infty)$ and
\[
\mathcal{F} := \{B(x,r) : x\in X, r>0\}, \quad
\xi(x,r)=1, \quad
\eta(x,r)=\psi(x,r) = r.
\]
Comparing \eqref{eqn:malpha0} and \eqref{eqn:mCZa}, we see that $m_C(Z,\alpha)=m_H(Z,\alpha)$ for every $Z\subset X$, and the Hausdorff dimension $\dim_H(Z)$ is the critical value such that $m_H(Z,\alpha)$ is infinite for
$\alpha <\dim_H(Z)$ and $0$ for $\alpha >\dim_H(Z)$.  Thus $\dim_CZ=\dim_HZ$, and the outer measure $m_C(\cdot,\dim_H Z)$ on $Z$ is the $(\dim_H Z)$-dimensional spherical Hausdorff measure.

It is useful to understand when an outer measure defines a Borel measure on a metric space. Recall that an outer measure $m$ on a metric space $(X,d)$ is a \emph{metric outer measure} if $m(E\cup F) = m(E) + m(F)$ whenever $d(E,F) := \inf\{d(x,y) : x\in E, y\in F)\}>0$.

\begin{proposition}[{\cite[\S2.3.2(9)]{hF69}}]\label{prop:metric-borel}
If $X$ is a metric space and $m$ is a metric outer measure on $X$, then every Borel set in $X$ is $m$-measurable, and so $m$ defines a Borel measure on $X$.
\end{proposition}

Given any $E,F \subset X$, with $d(E,F)>0$, we see that any cover $\mathcal{G}\subset \mathcal{F}$ of $E\cup F$ with $\psi(\mathcal{G}) \leq d(E,F)/2$ can be written as the disjoint union of a cover of $E$ and a cover of $F$; using this it is easy to show that $m_H(E\cup F,\alpha) = m_H(E,\alpha) + m_H(F,\alpha)$, so $m_H(\cdot,\alpha)$ is a metric outer measure.  By Proposition \ref{prop:metric-borel}, this defines a Borel measure on $X$.

\subsubsection{Topological pressure as a Carath\'eodory dimension}\label{sec:pres} 
Let $f$ be a continuous map of a compact metric space $X$, and $\ph\colon X\to \RR$ a continuous function. Then as described already in \S\ref{sec:car}, one can consider covers that are refined dynamically rather than geometrically.  This was done first by Bowen to define topological entropy in a more general setting \cite{rB73}, and then extended by Pesin and Pitskel' to topological pressure \cite{PP84}.  Here we give a definition that differs slightly from \cite{PP84} but gives the same dimensional quantity \cite[Proposition 5.2]{vC11}.

Fix $r>0$ and to each $(x,n) \in X\times \NN$, associate the Bowen ball $B_n(x,r)$.  Let $\mathcal{F}$ be the collection of all such Bowen balls,
and let $\mathcal{S} = X\times \NN$, so $s=(x,n)$ has $U_s = B_n(x,r)$.  Now put
\begin{equation}\label{eqn:PC}
\xi(x,n) = e^{S_n\ph(x)}, \quad \eta(x,n) = e^{-n}, \quad \psi(x,n) = \tfrac 1n.
\end{equation}
It is easy to see that $(\mathcal{S},\mathcal{F},\xi,\eta,\psi)$ satisfies \ref{A1}--\ref{A3}, so this defines a $C$-structure.  The associated outer measure is given by
\begin{equation}\label{eqn:mCP}
m_C(Z,\alpha) = \lim_{N\to\infty} \inf_\mathcal{G} \sum_{(x,n)\in \mathcal{G}} e^{S_n\ph(x)} e^{-n\alpha},
\end{equation}
where the infimum is over all $\mathcal{G} \subset \mathcal{S}$ such that $\bigcup_{(x,n)\in \mathcal{G}} B_n(x,r) \supset Z$ and $n\geq N$ for all $(x,n)\in \mathcal{G}$.

\begin{remark}\label{rmk:not-metric}
The measure $m_C(\cdot,\alpha)$ is not necessarily a metric outer measure, since there may be $x\neq y\in X$ such that $y\in B_n(x,r)$ for all $n\in \NN$.\footnote{In fact $m_C(\cdot,\alpha)$ is an outer measure if and only if $f$ is positively expansive to scale $r$.}  Thus Borel sets in $X$ need not be $m_C(\cdot,\alpha)$-measurable.
\end{remark}

Writing $\dim_C^r Z$ for the critical value of $\alpha$, where the superscript emphasizes the dependence on $r$, the quantity
\[
P_Z(\ph)=\ulim_{r\to 0} \dim_C^r Z
\]
is called the the \emph{topological pressure} of $\ph$ on the set $Z$. Observe that this notion of the topological pressure is more general than the one introduced in 
Definition \ref{def:pressure} as it is more suited to
arbitrary subsets $Z$ (which need not be compact or invariant); both definitions agree when $Z=X$ \cite[Theorem 11.5]{pes97}.

\subsubsection{A $C$-structure on local unstable leaves}\label{car.measure}
Now consider the setting of Theorems \ref{thm:finite}--\ref{thm:holonomy}: $\Lambda$ is a hyperbolic set for a $C^{1+\alpha}$ diffeomorphism $f$, and $\ph\colon \Lambda\to \RR$ is H\"older continuous. Fix $r>0$ and define a $C$-structure on $X = \Vl^u(x)\cap\Lambda$, which depends on $\ph$, in the following way. To each $(x,n)\in X\times\NN$, associate the Bowen ball $B_n(x,r)$. Let $\mathcal{F}$ be the collection of all such balls, and let 
$\mathcal{S} = X\times\NN$, so $s=(x,n)$ has $U_s=B_n(x,r)$.  Now put 
\begin{equation}\label{eqn:PC2}
\xi(x,n) = e^{S_n\ph(x)}, \quad \eta(x,n) = e^{-n}, \quad \psi(x,n) = \tfrac 1n.
\end{equation}
Again, $(\mathcal{S},\mathcal{F},\xi,\eta,\psi)$ satisfies \ref{A1}--\ref{A3} and defines a $C$-structure, whose associated outer measure is given by
\begin{equation}\label{eqn:mCP2}
m_C(Z,\alpha)=\lim_{N\to\infty}\inf_\mathcal{G}\sum_{(x,n)\in\mathcal{G}}e^{S_n\ph(x)} 
e^{-n\alpha},
\end{equation}
where the infimum is over all $\mathcal{G} \subset \mathcal{S}$ such that 
$\bigcup_{(x,n)\in\mathcal{G}}B_n(x,r)\supset Z$ and $n\geq N$ for all 
$(x,n)\in\mathcal{G}$.

Given $x\in \Lambda$ we are interested in two things: 
\begin{enumerate}
\item the Carath\'eodory dimension of $X$, as determined by this $C$-structure; and
\item the (outer) measure on $X$ defined by \eqref{eqn:mCP} at $\alpha = \dim_C(X)$.
\end{enumerate}
The first of these is settled by the following, which is proved in \cite[Theorem 4.2(1)]{CPZ}.
%we prove in \S\ref{sec:finite}.

\begin{proposition}\label{prop:P-same}
With $\Lambda,f,\ph,r$ as above, and the $C$-structure defined on $X = \Vl^u(x)\cap \L$ by Bowen balls $B_n(x,r)$ and \eqref{eqn:PC2}, we have $\dim_C^r(X) = P(\ph)$ for every $x\in \Lambda$.  In particular, this implies that $P_X(\ph) = P_\Lambda(\ph)$.
\end{proposition}

Note that on each $X = \Vl^u(x)\cap \Lambda$, covers by Bowen balls $B_n(x,r)$ are the same thing as covers by $u$-Bowen balls $B_n^u(x,r) = B_n(x,r) \cap \Vl^u(x)$, which we used in \S\ref{sec:car}.  Thus when we put $\alpha = P(\ph)$, we see that \eqref{eqn:mCP2} agrees with \eqref{car0} for every $Z \subset X$, and in particular, the quantity $m_x^\C(Z)$ defined in \eqref{car0} is the outer measure on $X$ associated to the $C$-structure above and the parameter value $\alpha=P(\ph)$.  

One must still do some work to show that this outer measure is finite and nonzero; this is done in \cite{CPZ}, and the idea of the argument is given in \S\ref{sec:measures} below.
%we do this in \S\S\ref{sec:measures}.  
We conclude this section by observing that the issue raised in Remark \ref{rmk:not-metric} is not a problem here, and that we have in fact defined a \emph{metric} outer measure.  Indeed, given any $x\in \Lambda$ and $y\in \Vl^u(x) \cap \Lambda$, we have $\diam B_n^u(y,r) \leq r \lambda^n$ for all $n\in \NN$ by Proposition \ref{prop:local-mfds}, so if $E,F\subset X$ have $d(E,F)>0$, then there is $N\in \NN$ such that $B_n(y,r) \cap B_k(z,r) = \emptyset$ whenever $y\in E$, $z\in F$, and $k,n\geq N$.  
Then for $N$ sufficiently large, any $\mathcal{G}$ as in \eqref{eqn:mCP2} has the property that it splits into disjoint covers of $E$ and $F$, and thus $m_x^\C(E \cup F) = m_x^\C(E) + m_x^\C(F)$.  By Proposition \ref{prop:metric-borel}, $m_x^\C$ defines a Borel measure on $X$, as claimed in Theorem \ref{thm:finite}.

\subsection{An application: measures of maximal dimension}
If $X$ is a measurable space with a measure $\mu$, and $\dim_C$ is a Carath\'eodory dimension on $X$, then the quantity
\begin{align*}
\dim_C\mu&=\inf\{\dim_CZ\colon \mu(Z)=1\}\\
&=\lim_{\delta\to 0}\inf\{\dim_CZ\colon \mu(Z)>1-\delta\}
\end{align*}
is called the \emph{Carath\'eodory dimension} of $\mu$. We say that $\mu$ is a \emph{measure of maximal Carath\'eodory dimension} if
$\dim_C\mu=\dim_CX$. Note that if the Carath\'eodory measure 
$m_C(X,\alpha)$ at dimension $\alpha=\dim_CX$ is finite and positive, then this measure is a measure of maximal Carath\'eodory dimension.

With $f\colon \L\to\L$ as in Theorem \ref{thm:finite},
we consider a particular but important family of potential functions 
$\phigeo_t(x)$ on $\L$, called the \emph{geometric $t$-potentials}: for any $t\in\mathbb{R}$
\[
\phigeo_t(x):=-t\log |\det Df|_{E^u(x)}|.
\] 
Since the subspace $E^u(x)$ depends H\"older continuously on $x\in\L$ (see \eqref{holder}), for each $t\in\mathbb{R}$ the function $\phigeo_t(x)$ is H\"older continuous and hence, it admits a unique equilibrium state $\mu_t:=\mu_{\phigeo_t}$. 

We consider the function $P(t):=P(\phigeo_t)$ called the \emph{pressure function}. One can show that this function is monotonically decreasing, convex and real analytic in $t$. Moreover, $P(t)\to +\infty$ as $t\to-\infty$ and $P(t)\to-\infty$ as $t\to+\infty$ with 
$P(1)\le 0$. Therefore, there is a number $0<t_0\le 1$ which is the unique solution of \emph{Bowen's equation} $P(t)=0$. We shall show that given $x\in\L$, there is a $C$-structure on the set $X=\Vl^u(x)\cap\L$ with respect to which $t_0$ is the Carath\'eodory dimension of the set $X$. Indeed, since $P(t_0)=0$, the measure $m_x^{t_0} :=m_x^\C$, given by \eqref{car0} for $\C = (\phigeo_{t_0},r)$, can be written as  
\begin{equation}\label{t0}
m_x^{t_0}(Z)=\lim_{N\to\infty}\inf\left\{\sum_i\bigg(\prod_{k=0}^{n_i-1}\det\big(Df|_{E^u(f^k(x_i))}\big)\bigg)^{-t_0}\right\},
\end{equation}
where the infimum is taken over all collections $\{B^u_{n_i}(x_i,r)\}$ of $u$-Bowen balls with $x_i\in X$, $n_i\ge N$, which cover $Z$.

Relation \eqref{t0} shows that the measure $m_x^{t_0}$ is the Carath\'eodory measure generated by the $C$-structure $\tau'=(\mathcal{S},\mathcal{F},\xi',\eta',\psi)$,
where
\[
\xi'(x,n) := 1, \qquad \eta'(x,n) := \prod_{k=0}^{n_i-1}\det\big(Df|_{E^u(f^k(x_i))}\big)^{-1}.
\]
It is easy to see that with respect to the $C$-structure $\tau'$ we have that
$\dim_{C,\tau'}X=t_0$ and the measure $m_x^{t_0}=m_{C,\tau'}(\cdot, t)$ is the measure of maximal Carath\'eodory dimension. In particular, the Carath\'eodory dimension of $X=\Vl^u(x)\cap\L$ does not depend on the choice of the point $x\in\L$. It is also clear that the number $t_0$ depends continuously on $f$ in the $C^1$ topology and hence, so does the Carath\'eodory dimension $\dim_{C,\tau'}X$.

We consider the particular case when the map $f$ is $u$-conformal; that is 
$Df_{|E^u(x)}=a(x)\mathrm{Isom}_x$ for all $x\in\L$, where $\mathrm{Isom}_x$ is an isometry. The direct calculation involving \eqref{t0} shows that in this case $m_x^{t_0}$ is a measure of full Hausdorff dimension and that $t_0\dim E^u=\dim_HX$. 

Given a locally maximal hyperbolic set $\Lambda$, it has been a long-standing open problem to compute the Hausdorff dimension of the set $X=\Vl^u(x)\cap\L$ and to find an invariant measure whose conditional measures on unstable leaves have maximal Hausdorff dimension, provided such a measure exists.
The above result solves this problem for $u$-conformal diffeomorphisms. The reader can find the original proof and relevant references in \cite{pes97}. 
It was recently proved that without the assumption of $u$-conformality, there are examples for which there is no invariant measure whose conditionals have full Hausdorff dimension; see \cite{DS}.
%It is however, known that for maps $f$ that are not $u$-conformal there may be no \emph{invariant} measure for which the conditional measures it generates on unstable leaves have full Hausdorff dimension, see \cite{DS}. Thus 
Theorem \ref{thm:main} provides one way to settle the issue in the non-conformal case by replacing `measure of maximal Hausdorff dimension' with `measure of maximal Carath\'eodory dimension' with respect to the $C$-structure $\tau'$ just described.

\section{Outline of proofs}\label{sec:proofs}

In \S\S\ref{sec:measures}--\ref{sec:scaling} we outline the proofs of Theorems \ref{thm:finite}--\ref{thm:holonomy}, referring to \cite{CPZ} for complete details; see Remarks 2.3 and 4.1 of that paper for an explanation of why the setting here is covered.  In \S\ref{sec:geometric} we prove Theorem \ref{thm:main}, again referring to \cite{CPZ} for certain technicalities.  In \S\ref{sec:push-cond} we give a complete proof of Theorem \ref{push-cond1}.

\subsection{Reference measures are nonzero and finite}\label{sec:measures}

Recall that $\Lambda$ is a locally maximal hyperbolic set for $f$, on which each $x$ has local stable and unstable manifolds of size $\tau>0$. We assume that $f|_\Lambda$ is topologically transitive. In what follows we occasionally use the following notation: given 
$A,B,C,a\ge 0$, we write $A=C^{\pm a}B$ as shorthand to mean 
$C^{-a}B\le A\le C^a B$. 
The key to the proof of Theorem \ref{thm:finite} is the following result.

\begin{proposition}\label{prop:uniform}
For every $r_1\in (0,\tau)$ and $r_2 \in (0,\tau/3]$ there is $C>1$ such that for every $x\in \Lambda$ and $n\in\NN$ we have
\begin{equation}\label{eqn:uniform}
\Zspan_n(B_\L^u(x,r_1),\ph,r_2) =  C^{\pm 1} e^{nP(\ph)}.
\end{equation}
\end{proposition}

Similar partition sum bounds are obtained in Bowen's paper \cite{rB745}, where they are proved for all of $\L$ instead of for a single unstable leaf. For the full proof of Proposition \ref{prop:uniform}, see \cite[\S6]{CPZ}; we outline the argument below.  As in Bowen's case, the underlying mechanism is a set of elementary lemmas, which we give in \S\ref{sec:lemmas}.  In \S\ref{sec:multiplicative} we explain why it is reasonable to expect these lemmas to apply to the sequence $\Zspan_n$, and in \S\ref{sec:complete-finite} we outline how Proposition \ref{prop:uniform} leads to Theorem \ref{thm:finite}.

\subsubsection{Elementary counting lemmas}\label{sec:lemmas}

\begin{lemma}\label{lem:fekete}
If $Z_n>0$ is a sequence of numbers satisfying $Z_{n+m} \leq Z_n Z_m$ for all $m,n$, then $P = \lim_{n\to\infty} \frac 1n \log Z_n$ exists and is equal to $\inf_{n\in\NN} \frac 1n \log Z_n$.  In particular, $Z_n \geq e^{nP}$ for every $n$.
\end{lemma}
\begin{proof}
Fix $n\in \NN$; then for all $m\in \NN$ we can write $m=an+b$ where $a\in \NN$ and $b\in \{0,1,\dots, n-1\}$, and iterate the submultiplicativity property to obtain $Z_m \leq Z_n^a Z_b$.  Taking logs and dividing by $m$ gives
\[
\frac 1m \log Z_m \leq \frac am \log Z_n + \frac {\log Z_b}m
\leq \frac{an}{m} \cdot \frac 1n \log Z_n + \frac{\max\{\log Z_0, \dots \log Z_{n-1}\}}m.
\]
Sending $m\to\infty$ we see that $\frac{an}m\to 1$, so
\begin{equation}\label{eqn:ZmZn}
\varlimsup_{m\to\infty} \frac 1m \log Z_m \leq \frac 1n \log Z_n.
\end{equation}
Since $n$ was arbitrary we deduce that
\[
\varlimsup_{m\to\infty} \frac 1m \log Z_m
\leq \inf_{n\in\NN} \frac 1n \log Z_n \leq \varliminf_{n\to\infty} \frac 1n \log Z_n,
\]
whence all three terms are equal and the limit exists.  Now \eqref{eqn:ZmZn} implies that $Z_n \geq e^{nP}$.
\end{proof}

\begin{lemma}\label{lem:mult-bound}
If $Z_n>0$ is a sequence of numbers satisfying
$Z_{n+m} \leq C Z_n Z_m$ for all $m,n$, where $C>0$ is independent of $m,n$, then $P = \lim_{n\to\infty} \frac 1n \log Z_n$ exists and is equal to $\inf_{n\in\NN} \frac 1n \log(CZ_n)$.  In particular, $Z_n \geq C^{-1} e^{nP}$ for all $n$.
\end{lemma}
\begin{proof}
Follows by applying Lemma \ref{lem:fekete} to the sequence $Y_n = CZ_n$, which satisfies $Y_{n+m} = CZ_{n+m} \leq C^2 Z_n Z_m = Y_n Y_m$.
\end{proof}

\begin{lemma}\label{lem:mult-bound-2}
If $Z_n>0$ is a sequence of numbers satisfying
$Z_{n+m} \geq C^{-1} Z_n Z_m$ for all $m,n$, where $C>0$ is independent of $m,n$, then $P = \lim_{n\to\infty} \frac 1n \log Z_n$ exists and is equal to $\sup_{n\in\NN} \frac 1n \log(Z_n/C)$.  In particular, $Z_n \leq C e^{nP}$ for all $n$.
\end{lemma}
\begin{proof}
Follows by applying Lemma \ref{lem:fekete} to the sequence $Y_n = C/Z_n$, which satisfies $Y_{n+m} = C/Z_{n+m} \leq C^2/(Z_n Z_m) = Y_n Y_m$.
\end{proof}

\subsubsection{Partition sums are nearly multiplicative}\label{sec:multiplicative}

In light of Lemmas \ref{lem:mult-bound} and \ref{lem:mult-bound-2}, Proposition \ref{prop:uniform} can be proved by showing that the partition sums $\Zspan_n(B_\L^u(x,r_1),\ph,r_2)$ are `nearly multiplicative': $\Zspan_{n+m} = C^{\pm 1} \Zspan_n \Zspan_m$.
A short argument given in \cite[Lemma 6.3]{CPZ} shows that $\Zsep_n = e^{\pm Q_u} \Zspan_n$, and thus it suffices to show that
\[
\Zspan_{n+m} \leq C \Zspan_n \Zspan_m,
\qquad
\Zsep_{n+m} \geq C^{-1} \Zsep_n \Zsep_m,
\]
where we are being deliberately vague about the arguments of $\Zspan$ and $\Zsep$.

\begin{figure}[htbp]
\includegraphics[width=.85\textwidth]{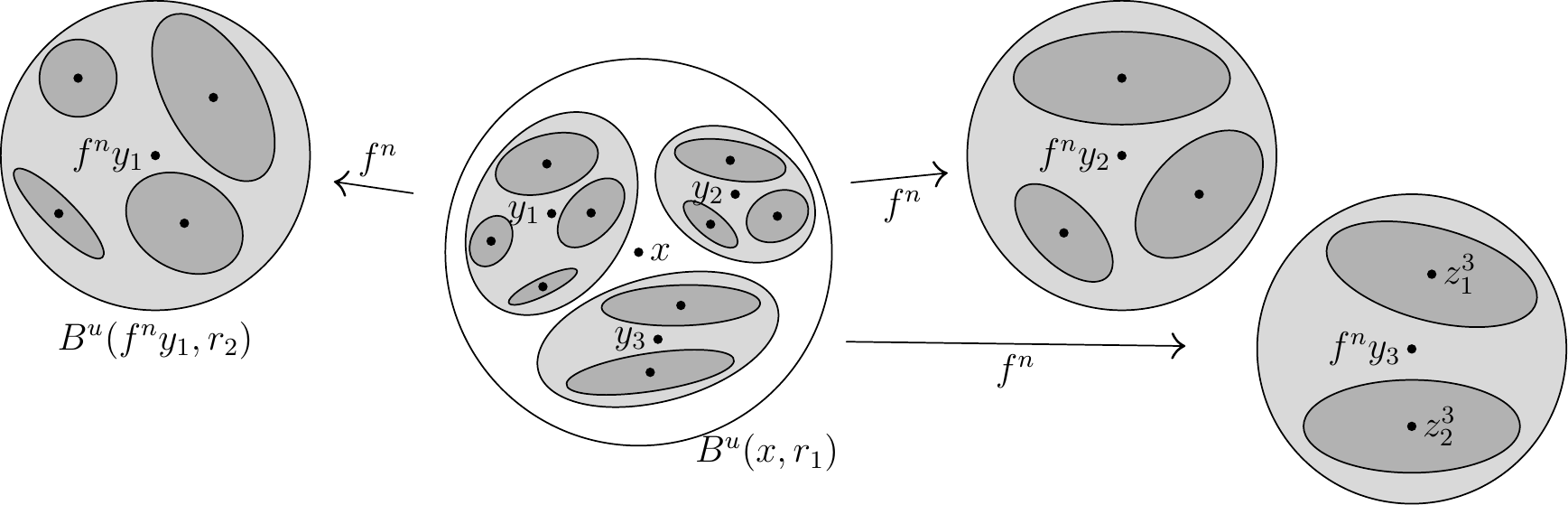}
\caption{Proving that $\Zsep_{n+m} \geq C^{-1} \Zsep_n \Zsep_m$.}
\label{fig:sepsep}
\end{figure}

Figure \ref{fig:sepsep} illustrates the idea driving the estimate for $\Zsep$: if $E_n = \{y_1,\dots, y_a\} \subset B_\L^u(x,r_1)$ is a maximal $(n,r_2)$-separated set of points, and to each $1\leq i\leq a$ we associate a maximal $(m,r_2)$-separated set $E_m^i = \{z_1^i,\dots, z_{b_i}^i\} \subset B_\L^u(f^n y_i, r_2)$, then pulling back all the points $z_j^i$ gives an $(m+n,r_2)$-separated set
\[
E_{m+n} = \bigcup_{i=1}^a f^{-n}(E_m^i) = \{f^{-n} z_j^i : 1\leq i\leq a, 1\leq j \leq b_i\} \subset B_\L^u(x,r_1),
\]
so we expect to get an estimate along the lines of
\begin{multline}\label{eqn:rough}
\Zsep_{m+n} \geq \sum_{p\in E_{m+n}} e^{S_{m+n} \ph(p)}
= \sum_{p\in E_{m+n}} e^{S_n \ph(p)} e^{S_m\ph(f^n(p))}
\gtrapprox \sum_{i=1}^a \sum_{j=1}^{b_i} e^{S_n \ph(y_i)} e^{S_m\ph(z_j^i)} \\
= \sum_{i=1}^a e^{S_n\ph(y_i)} \bigg( \sum_{j=1}^{b_i} e^{S_m \ph(z_j^i)} \bigg)
\approx \bigg( \sum_{i=1}^a e^{S_n\ph(y_i)} \bigg) \Zsep_m \approx \Zsep_m \Zsep_n,
\end{multline}
where we continue to be deliberately vague about the arguments of $\Zsep$.  If we can make this rigorous, then a similar argument with spanning sets instead of separated sets will lead to $\Zspan_{m+n} \leq C \Zspan_m \Zspan_n$, which will prove Proposition \ref{prop:uniform}.

But how do we make \eqref{eqn:rough} rigorous?  There are two sources of error which are hinted at by the ``$\approx$'' symbols.
\begin{enumerate}
\item Given  $p\in E_{m+n}$ and the corresponding $y\in E_n$ (where $f^n(p)\in B_\L^u(f^n y,r_2)$), the approximation on the first line of \eqref{eqn:rough} requires us to compare the ergodic sums $S_n\ph(p)$ and $S_n\ph(y)$.  In particular, we must find a constant $Q_u$ (independent of $y,p,n$) such that $|S_n\ph(p) - S_n\ph(y)| \leq Q_u$ whenever $p \in B_n^u(y,r_2)$.
\item The omission of the arguments for $\Zsep_n$ obscures the fact that $\Zsep_{m+n}$ and $\Zsep_m$ in \eqref{eqn:rough} both refer to $(n,r_2)$-separated subsets of $B_\L^u(x,r_1)$, while $\Zsep_n$ refers to $(n,r_2)$-separated subsets of $B_\L^u(f^n y_i, r_2)$.  Thus we must control how $\Zsep_n(B_\L^u(x,r_1),\ph,r_2)$ changes when we fix $n$ and let $x,r_1,r_2$ vary; in particular, we must find for each $r_1,r_1',r_2,r_2'$ a constant $C$ such that for every $n,x,y$, we have
\[
\Zsep_n(B_\L^u(x,r_1),\ph,r_2) = C^{\pm 1}\Zsep_n(B_\L^u(y,r_1'),\ph,r_2').
\]
\end{enumerate}

The first source of error described above can be controlled by establishing a generalized version of property \eqref{eqn:bdd-dist}. 

\begin{definition}\label{def:u-s-Bowen}
We say that a potential $\ph\colon \Lambda\to\RR$ has the \emph{$u$-Bowen property} if there is $Q_u>0$ such that for every $x\in \Lambda$, $n\geq 0$, and $y\in B_n^u(x,\tau) \cap \Lambda$, we have $|S_n\ph(x) - S_n\ph(y)| \leq Q_u$.  We also say that $\ph$ has the \emph{$s$-Bowen property} if there is $Q_s>0$ such that for every $x\in \Lambda$, $n\geq 0$, and $y\in B_\Lambda^s(x,\tau)=B^s(x,\tau)\cap\Lambda$, we have $|S_n\ph(x) - S_n\ph(y)| \leq Q_s$.\footnote{The asymmetry in the definition comes because $S_n\ph$ is a forward Birkhoff sum and $B_n^u(x,\tau)$ is defined in terms of forward iterates; one could equivalently define the $s$-Bowen property in terms of backward Birkhoff sums and $s$-Bowen balls.  The $s$-Bowen property is needed to control the second source of error described above.}
\end{definition}
\begin{lemma}\label{lem:bowen}
If $\ph\colon\Lambda\to\RR$ is H\"older continuous, then $\ph$ has the $u$-Bowen property and the $s$-Bowen property.\footnote{This is the only place where H\"older continuity is used; in particular, H\"older continuity could be replaced by the $u$- and $s$-Bowen properties in all our main results.}
\end{lemma}
\begin{proof}
We prove the $u$-Bowen property; the proof of the $s$-Bowen property is similar.  Given $y\in B_n^u(x,\tau)$, for every $0\leq k < n$ we have $d(x,y) \leq \tau \lambda^{n-k}$
where $0<\lambda <1$ is as in Proposition \ref{prop:local-mfds}\eqref{leaves-contract}, so writing $\beta$ for the H\"older exponent of $\ph$, we have
\begin{align*}
|S_n\ph(x) - S_n\ph(y)| &\leq \sum_{k=0}^{n-1} |\ph(f^k x) - \ph(f^k y)| 
\leq \sum_{k=0}^{n-1} |\ph|_\beta d(f^kx,f^ky)^\beta \\
&\leq |\ph|_\beta \tau^\beta \sum_{k=0}^{n-1} \lambda^{\beta(n-k)}
< |\ph|_\beta\tau^\beta(1-\lambda^\beta)^{-1} =: Q_u.\qedhere
\end{align*}
\end{proof}

To control the second source of error described above, the first main idea is that topological transitivity guarantees that for every $\delta>0$, the images $f^k(B_\L^u(y,\delta))$ eventually come within $\delta$ of $x$, and that the $k$ for which this occurs admits an upper bound that depends only on $\delta$.  Then given a spanning set $E \subset B_\L^u(y,\delta)$, the part of the image $f^k(E)$ that lies near $x$ can be moved by holonomy along stable manifolds to give a spanning set in the unstable leaf of $x$.  This is made precise in \cite[Lemma 6.4]{CPZ}.  One can use similar arguments to change the scales $r_1, r_2$; for example, if $x,y$ are on the same local unstable leaf and have orbits that remain within $r_2$ of each other until time $n$, then they remain within $r_2 \lambda^k$ of each other until time $n-k$.  See \cite[\S6]{CPZ} for full details.

\subsubsection{Proving Theorem \ref{thm:finite}}\label{sec:complete-finite}

Fix $x\in\Lambda$ and set $X:=\Vl^u(x) \cap \Lambda$.
We showed in \S\ref{car.measure} that $m_x^\C$ defines a metric outer measure on $X$, and hence gives a Borel measure. Note that the final claim in Theorem \ref{thm:finite} about agreement on intersections is immediate from the definition. Thus it remains to prove that $m_x^\C(X) \in [K^{-1},K]$, where $K$ is independent of $x$; this will complete the proof of Theorem \ref{thm:finite}, and will also prove Proposition \ref{prop:P-same}.

For full details, see \cite[\S6.5]{CPZ}.  The idea is that it suffices to prove that for a fixed $r>0$, we have $m_x^\C(B_\L^u(x,r))$ uniformly bounded away from $0$ and $\infty$, since each $\Vl^u(x)$ can be covered with a uniformly finite number of balls $B_\L^u(y,r)$.  The upper bound is easier to prove since it only requires that we exhibit a cover satisfying the desired inequality; this is provided by Proposition \ref{prop:uniform}, which guarantees existence of an $(n,r)$-spanning set $E_n\subset B_\L^u(x,r)$ such that 
\[
\sum_{y\in E_n} e^{S_n\ph(y)} \leq C e^{nP(\ph)},
\]
and thus \eqref{car0} gives
\[
m_x^\C(B_\L^u(x,r)) \leq \lim_{n\to\infty} \sum_{y\in E_n} e^{-nP(\ph)} e^{S_n\ph(y)} \leq C.
\]
The lower bound is a little trickier since we must obtain a lower bound for an \emph{arbitrary} cover by $u$-Bowen balls as in \eqref{car0}, which are allowed to be of different orders, so we do not immediately get an $(n,r)$-spanning set for some particular $n$.  This can be resolved by observing that any open cover of $\overline{B_\L^u(x,r)}$ has a finite subcover, so to bound $m_x^\C(\overline{B_\L^u(x,r)})$ it suffices to consider covers of the form $\{B_{n_i}^u(y_i,r) : 1\leq i\leq a, n_i \geq N\}$.  Given such a cover, one can take $n = \max(n_1,\dots, n_a)$ and use arguments similar to those in the proof of Proposition \ref{prop:uniform} to
cover each $B_{n_i}^u(y_i,r)$ by a union of $u$-Bowen balls $B_n^u(z_i^j,r)$ ($1\leq j \leq b_i$)
satisfying
\[
\sum_j e^{S_n\ph(z_i^j)} \leq C' e^{(n-n_i)P(\ph)} e^{S_{n_i}\ph(y_i)}
\]
for some constant $C'$ that is independent of our choice of covers.
Then the set $E = \{z_i^j : 1\leq i \leq a, 1\leq j\leq b_i\}$ is $(n,r)$-spanning for $B_\L^u(x,r)$ and satisfies
\[
C^{-1} e^{nP(\ph)} \leq \sum_{z\in E} e^{S_n\ph(z)}
\leq \sum_{i=1}^a C' e^{(n-n_i)P(\ph)} e^{S_{n_i}\ph(y_i)}.
\]
Dividing through by $e^{nP(\ph)}$, taking an infimum over all covers, and sending $N\to\infty$ gives $C^{-1} \leq C' m_x^\C(\overline{B_\L^u(x,r)})$.  Again, full details are in \cite[\S6.5]{CPZ}.

\subsection{Behavior of reference measures under iteration and holonomy}\label{sec:scaling}

\subsubsection{Iteration and the $u$-Gibbs property}

%Proof of Theorem \ref{thm:Gibbs}}\label{sec:Gibbs}

The simplest case of Theorem \ref{thm:Gibbs} occurs when $\ph=0$, so the claim is that $m_{f(x)}^\C = e^{P(0)} m_x^\C \circ f^{-1}$, which is exactly the scaling property satisfied by the Margulis measures on unstable leaves.  Given $E\subset \Vl^u(f(x))$, we see from the relationship $f^{-1}B_n^u(y,r) = B_{n+1}^u(f^{-1}y,r)$ that any cover $\{B_{n_i}^u(y_i,r)\}$ of $E$ leads immediately to a cover $\{B_{n_i+1}^u(f^{-1}y_i,r)\}$ of $f^{-1}E$, and vice versa.  Using this bijection in the definition of the reference measures in \eqref{car0}, we get
\[
m_{f(x)}^\C(E) = \lim_{N\to\infty} \inf \sum_i e^{-n_i P(0)} = e^{P(0)}\lim_{N\to\infty} \inf \sum_i e^{-(n_i+1) P(0)} = e^{P(0)} m_x^\C(f^{-1}E).
\]
For nonzero potentials one must account for the factor of $e^{S_{n_i}\ph(x_i)}$ in \eqref{car0}.  This can be done by partitioning $E$ into subsets $E_1,\dots, E_T$ on which $\ph$ is nearly constant, and repeating the above argument on each $E_i$ to get an approximate result that improves to the desired result as $T\to\infty$; see \cite[\S7.1]{CPZ} for details.

Once \eqref{eqn:scaling} has been proved, we can iterate it to obtain
%\subsubsection{Proof of Corollary \ref{cor:Gibbs}}\label{sec:gibbs}
%To prove Corollary \ref{cor:Gibbs}, we 
\begin{equation}\label{eqn:iterated-scaling}
m_{f^n(x)}^\C(A) = \int_{f^{-n}(A)} e^{nP(\ph) - S_n\ph(y)}\,dm_x^\C(y)
\end{equation}
for all $A\subset \Vl^u(f^n(x))$.  Applying this to $A=B^u(f^n(x),\delta) = f^n(B_n^u(x,\delta))$ and using Theorem \ref{thm:finite} gives a constant $Q_4=Q_4(\delta)$ such that
\begin{align*}
m_x^\C(B_n^u(x,\delta)) e^{nP(\ph) - S_n\ph(x)} &= e^{\pm Q_u} \int_{B_n^u(x,\delta)} e^{nP(\ph)-S_n\ph(y)}\,dm_x^\C(y) \\
&=e^{\pm Q_u} m_{f^n(x)}^\C(B^u(f^n(x),\delta)) =
e^{\pm Q_u} Q_4^{\pm 1},
\end{align*}
for every $x,n$, 
where the first estimate uses 
%and the integrand satisfies $e^{nP(\ph) - S_n\ph(y)} = e^{\pm Q_u} e^{nP(\ph) - S_n\ph(x)}$ by 
the $u$-Bowen property from Lemma \ref{lem:bowen}.  
This establishes the $u$-Gibbs property for $m_x^\C$ with $Q_1 = Q_4 e^{Q_u}$ and proves Corollary \ref{cor:Gibbs}.

%\subsubsection{Proof of Theorem \ref{thm:holonomy}}\label{sec:holonomy}
\subsubsection{Holonomy maps}

Given nearby points $y,z$ and sets $E_y\subset \Vl^u(y)$, $E_z \subset \Vl^u(z)$ such that $\pi_{yz}(E_y) = E_z$ (with respect to some rectangle), we observe that every cover of $E_y$ by $u$-Bowen balls $\{B_{n_i}^u(x_i,r)\}$ produces a cover of $E_z$ by the images $\{\pi_{yz}B_{n_i}^u(x_i,r)\}$.
If $y,z$ are close enough to each other to guarantee that
\begin{equation}\label{eqn:piyz}
\pi_{yz}B_{n_i}^u(x_i,r) \subset B_{n_i}^u(x_i,2r)
\end{equation}
for each $i$,
then we get $E_z \subset \bigcup_i B_{n_i}^u(x_i,2r)$.  
Fixing $k\in \NN$ such that each $x\in \L$ has $B_\L^u(x,2r) \subset \bigcup_{j=1}^k B_\L^u(x^j,r)$ for some $x^1,\dots, x^k$, we see that $E_z \subset \bigcup_{i,j} B_{n_i}^u(x_i^j,r)$, and thus \eqref{car0} gives
\[
m_z^\C(E_z) \leq \sum_{i,j} e^{-n_i P(\ph)} e^{S_{n_i}\ph(x_i^j)}
\leq k \sum_i e^{-n_i P(\ph)} e^{S_{n_i}\ph(x_i) + Q_u};
\]
taking an infimum and then a limit gives $m_z^\C(E_z) \leq k e^{Q_u} m_y^\C(E_y)$.

In general, if $y,z$ lie close enough for holonomy maps to be defined, but not close enough for \eqref{eqn:piyz} to hold, then we can iterate $E_y,E_z$ forward until some time $n$ at which $f^ny,f^nz$ are close enough for the previous part to work, and use Theorem \ref{thm:Gibbs} to get (assuming without loss of generality that $E_y \subset B_n^u(y,\tau)$, and similarly for $E_z$)
\begin{multline*}
m_z^\C(E_z) = \int_{f^n(E_z)} e^{-nP(\ph) + S_n\ph(f^{-n}x)}\,dm_{f^n z}^\C(x) \\
\leq k e^{Q_u} \int_{f^n(E_y)} e^{-nP(\ph) + S_n\ph(f^{-n}x') + Q_s}\,dm_{f^n y}^\C(x')
= k e^{Q_u + Q_s} m_y^\C(E_y),
\end{multline*}
where the inequality uses the result from the previous paragraph.  Since the roles of $y,z$ were symmetric, this proves Theorem \ref{thm:holonomy} with $Q_2 = ke^{Q_u + Q_s}$.  See \cite[\S7.3]{CPZ} for a more detailed version of this argument.

\subsection{Geometric construction of equilibrium states}\label{sec:geometric}

Now that we have established the basic properties of the reference measures $m_x^\C$ associated to a H\"older continuous potential function $\ph$, the steps in the geometric construction of the unique equilibrium state $\mu_\ph$ are as follows.

\begin{enumerate}
\item Prove that every weak*-limit point $\mu$ of the sequence of probability measures $\mu_n = \frac 1n \sum_{k=0}^{n-1} f_*^k m_x^\C / m_x^\C(\Vl^u(x))$ %from \eqref{u-evolution} 
is an invariant measure whose conditional measures satisfy part \ref{3} of Theorem \ref{thm:main}; in particular, they are equivalent to the reference measures $m_y^\C$.
\item Use this to deduce that any such $\mu$ satisfies part \ref{2} of Theorem \ref{thm:main}, namely:
\begin{enumerate}[label=\upshape{(\alph{*})}]
\item\label{2a} the conditional measures of $\mu$ are absolutely continuous with respect to stable holonomies, and therefore $\mu$ is ergodic by the Hopf argument (Proposition \ref{prop:hopf}); 
\item\label{2b} $\mu$ gives positive weight to every open set in $\L$; 
\item\label{2c} the $u$-Gibbs property of the reference measures implies the Gibbs property \eqref{gibbs2} for $\mu$; and
\item\label{2d} $\mu$ is the unique equilibrium state for $\ph$ by Proposition \ref{prop:unique}.
\end{enumerate}
\item Observe that each $\mu_n$ is a Borel probability measure on $\L$, 
%for every choice of $x\in \L$, each member of the sequence $\mu_n$ is a Borel measure on $\L$ with total weight $\mu_n(\L) = m_x^\C(\Vl^u(x)) \in [K^{-1},K]$, 
and thus every subsequence has a subsubsequence that converges in the weak*-topology to a Borel probability measure $\mu$, which must be the unique equilibrium state $\mu_\ph$ by the previous step.
% with $\mu(\L) = m_x^\C(\Vl^u(x))$.  By the previous steps, the normalization of this measure is the unique equilibrium state $\mu_\ph$.  
Since every subsequence of $\mu_n$ has a subsubsequence converging to $\mu_\ph$, it follows that the sequence itself converges to this limit, which establishes part \ref{1} of Theorem \ref{thm:main}.
\end{enumerate}

The first step takes most of the work; once it is done, parts \ref{2a}--\ref{2c} of the second step only require short arguments that leverage the properties already established, part \ref{2d} of the second step merely consists of observing that $\mu$ satisfies the hypotheses of Proposition \ref{prop:unique}, and the third step is completely contained in the paragraph above.  Thus we outline here the argument for the first step and parts \ref{2a}--\ref{2c} of the second step, referring once more to \cite{CPZ} for complete details.

\subsubsection{Conditional measures of limiting measures}

In order to understand the conditional measures of $\mu = \lim_{k\to\infty} \mu_{n_k}$, we start by studying the conditional measures of $\mu_n$.
Given $x\in \L$ and $n\in \NN$, the iterate $f_*^n m_x^\C$ is supported on $W_n = f^n(\Vl^u(x) \cap \L)$, and given any $y\in W_n$, we can iterate the formula from Theorem \ref{thm:Gibbs} and obtain
\begin{equation}\label{eqn:RN}
\frac{d((f_*^n m_x^\C)|_{\Vl^u(y)})}{dm_y^\C}(z) = e^{-nP(\ph) + S_n\ph(f^{-n} z)} =: g_n(z)
\end{equation}
for every $z\in W_n \cap \Vl^u(y)$.
One can show that $g_n \to 0$ as $n\to\infty$, so it is convenient to write $\rho_n^y(z) := g_n(z) / g_n(y)$, and Lemma \ref{lem:bowen} gives
\begin{equation}\label{eqn:eQu}
\rho_n^y(z) = e^{S_n\ph(f^{-n} z) - S_n\ph(f^{-n} y)} \in [e^{-Q_u},e^{Q_u}].
\end{equation}
These functions describe the conditional measures of $f_*^n m_x^\C$.  Indeed, given a rectangle $R$, choose $y_1,\dots, y_a \in f^n(\Vl^u(x)) \cap R$ such that $f^n(\Vl^u(x)) \cap R \subset \bigcup_{i=1}^a \Vl^u(y_i)$, as in Figure \ref{fig:fnmxC}.  Then for every Borel set $E\subset R$, we have\footnote{There is a small technical issue here, namely that there may be some $y_i$ at which $W_n$ does not cross $R$ completely, and so the integral in \eqref{eqn:convex} actually gives too large a value.  However, this can only occur if $z_i = f^{-n}(y_i)$ is very close to the boundary of $\Vl^u(x)$, and the contribution of such points is negligible in the limit; see \cite{CPZ}.}
\begin{equation}\label{eqn:convex}
f^n_*m_x^\C(E)=\sum_{i=1}^{a}\int_E\,g_n(z)\,d m_{y_i}^\C(z)  
=\sum_{i=1}^{a} g_n(y_i)\int_E\rho_n^{y_i}(z)\,d m_{y_i}^\C(z).
\end{equation}
In other words, one can write $f_*^n m_x^\C|_R$ as a linear combination of the measures  $\rho_n^{y_i}\,dm_{y_i}^\C$ associated to the \emph{standard pairs}\footnote{Standard pairs consisting of a local leaf $\Vl^u(y)$ and a density function $\rho$ were introduced by Chernov and Dolgopyat in \cite{CherDolg} to study stochastic properties of dynamical systems; they are also used in constructing SRB measures for some dynamical systems with weak hyperbolicity, see \cite{CDP}.% Our construction of equilibrium states however exploits a different approach based on Proposition \ref{prop:conditionals}.
} $(\Vl^u(y_i), \rho_n^{y_i})$, with coefficients given by $g_n(y_i)$. 
This immediately implies that the conditional measures of $\mu_n$ on local unstable leaves are absolutely continuous with respect to the reference measures $m_x^\C$, with densities bounded away from $0$ and $\infty$.

\begin{figure}[htbp]
\includegraphics[width=.9\textwidth]{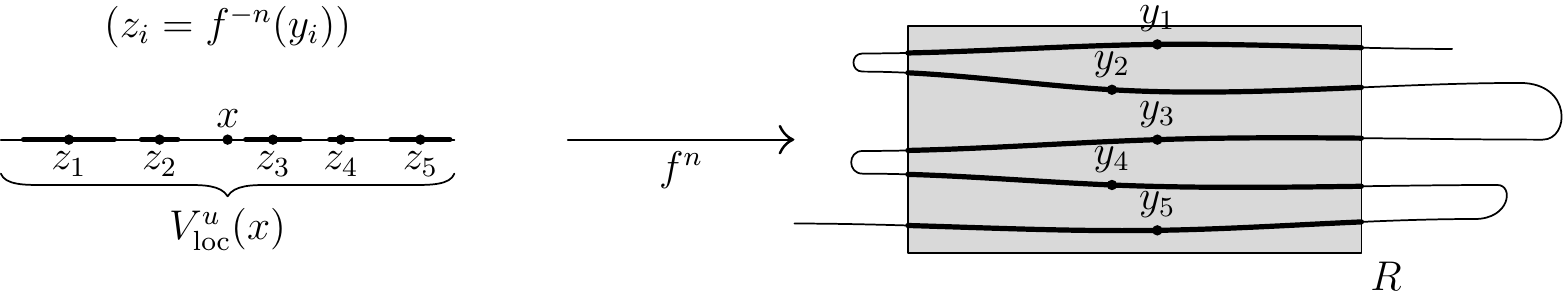}
\caption{Studying $f_*^n m_x^\C$ on a rectangle $R$.}
\label{fig:fnmxC}
\end{figure}

%Naively, one would now like to say something like ``since $\mu$ is the limit of the measures $\mu_{n_k}$, the conditional measures of $\mu$ are the limits of the conditional measures of $\mu_{n_k}$'' and then use the observation above to conclude the proof of Part \ref{3} of Theorem \ref{thm:main}.  Unfortunately, this is not true in general, 
 
To go further, we need the following characterization of the conditional measures, which is an immediate consequence of \cite[Corollary 5.21]{EW11}.\footnote{See
\cite{PS82} for the analogous argument controlling the conditionals of $\mu$ when $\ph$ is the geometric potential and the reference measure is leaf volume.  In that setting, the role of Proposition \ref{prop:conditionals} here is played by \cite[Lemma 13]{PS82}.}

\begin{proposition}\label{prop:conditionals}
Let $\mu$ be a finite Borel measure on $\L$ and let $R\subset \L$ be a rectangle with $\mu(R)>0$.  Let $\{\xi_\ell\}_{\ell\in \NN}$ be a refining sequence of finite partitions of $R$ that converge to the partition $\xi$ into local unstable sets $V_R^u(y) = \Vl^u(y)\cap R$.
Then there is a set $R' \subset R$ with $\mu(R')=\mu(R)$ such that for every $y\in R'$ and every continuous $\psi\colon R\to \RR$, we have
\begin{equation}\label{eqn:conditional}
\int_{V_R^u(y)} \psi(z) \,d %\mu_{V_R^u(y)}^\xi(z)
\mu_y^u(z)
= \lim_{\ell\to\infty} \frac 1{\mu(\xi_\ell(y))} \int_{\xi_\ell(y)} \psi(z)\,d\mu(z),
\end{equation}
where $\xi_n(y)$ denotes the element of the partition $\xi_\ell$ that contains $y$.
\end{proposition}

\begin{figure}[htbp]
\includegraphics[width=.8\textwidth]{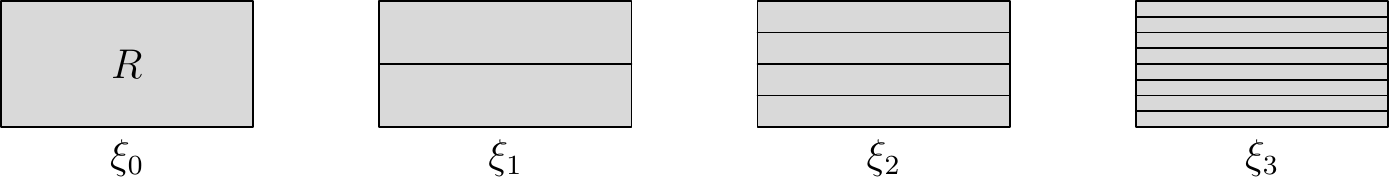}
\caption{A refining sequence of partitions of $R$.}
\label{fig:refining}
\end{figure}

Now the proof of Part \ref{3} of Theorem \ref{thm:main} goes as follows.  
Given a rectangle $R\subset \L$ with $\mu(R)>0$, let $\xi_\ell$ be a refining sequence of finite partitions of $R$ such that for every $y\in R$ and $\ell\in \NN$, the set $\xi_\ell(y)$ is a rectangle, and $\bigcap_{\ell\in \NN} \xi_\ell(y) = V_R^u(y)$, as in Figure \ref{fig:refining}.
Let $R'\subset R$ be the set given by Proposition \ref{prop:conditionals}.  We prove that \eqref{eqn:mu-m} holds for each $y\in R'$ by showing that for every positive continuous function $\psi\colon R\to \RR$, we have
\begin{equation}\label{eqn:mu-m-psi}
\int_{V_R^u(y)} \psi\,d\mu_y^u = \frac{Q_3^{\pm 1}}{m_y^\C(V_R^u(y))} \int_{V_R^u(y)} \psi\,d m_y^\C,
\end{equation}
where $Q_3$ is a constant that is independent of $\psi$.
To this end, we need to compare $\mu(\xi_\ell(y))$ and $\int_{\xi_\ell(y)} \psi(z)\,d\mu(z)$ and then apply \eqref{eqn:conditional}.
We see from \eqref{eqn:eQu} and \eqref{eqn:convex} that 
for each $j\in \NN$ there is a finite set $Y_j \subset R$ such that
%there is a finite set $\{ y_{j,i} : 0\leq j < n, 1\leq i \leq a_i \} \subset R$ such that 
%$f^j \Vl^u(x) \subset \bigcup_{i=1}^{a_i} \Vl^u(y_{j,i})$ for all $0\leq j < n$, and moreover,
%for all $0\leq j < n$, we have
\begin{equation}\label{eqn:int-xil}
\int_{\xi_\ell(y)} \psi \,d(f_*^n m_x^\C) = 
%\sum_{i=1}^{a_i} g_j(y_{j,i}) \int_{V_R^u(y_{j,i})} \psi(z) %\rho_j^{y_{j,i}}(z) 
%e^{\pm Q_u}
%\,dm_{y_{j,i}}^\C(z).
\sum_{p\in Y_j} g_j(p) \int_{V_R^u(p)} \psi(z) e^{\pm Q_u} \,dm_p^\C(z).
\end{equation}
Given $p\in \xi_\ell(y)$, Theorem \ref{thm:holonomy} gives
\[
\int_{V_R^u(p)} \psi \,dm_p^\C = Q_2^{\pm 1} \int_{V_R^u(y)} \psi(\pi_{py} z') \,dm_y^\C(z') = (2Q_2)^{\pm 1} \int_{V_R^u(y)} \psi \,dm_y^\C
\]
whenever $p,y$ are sufficiently close that $\psi(\pi_{py} z') = 2^{\pm 1} \psi(z')$ for all $z'\in V_R^u(y)$.  Thus for all sufficiently large $\ell$, \eqref{eqn:int-xil} gives
\[
\int_{\xi_\ell(y)} \psi\,d(f_*^n m_x^\C) = e^{\pm Q_u} (2Q_2)^{\pm 1} \bigg( \sum_{p\in Y_j} g_j(p) \bigg) \int_{V_R^u(y)} \psi\,dm_y^\C.
\]
Averaging over $0\leq j < n_k$ and sending $k\to\infty$ gives
\begin{equation}\label{eqn:int-1}
\int_{\xi_\ell(y)} \psi\,d\mu = (2Q_2e^{Q_u})^{\pm 1} \Big(\ulim_{k\to\infty} \frac 1{n_k} \sum_{j=0}^{n_k-1} \sum_{p\in Y_j} g_k(p)\Big) \int_{V_R^u(y)} \psi\,dm_y^\C.
\end{equation}
When $\psi\equiv 1$ this gives
\begin{equation}\label{eqn:int-2}
\mu(\xi_\ell(y)) = (2Q_2e^{Q_u})^{\pm 1} \Big(\ulim_{k\to\infty} \frac 1{n_k} \sum_{j=0}^{n_k-1} \sum_{p\in Y_j} g_k(p)\Big) m_y^\C(V_R^u(y)).
\end{equation}
Dividing \eqref{eqn:int-1} by \eqref{eqn:int-2}, sending $\ell\to\infty$, and using \eqref{eqn:conditional} yields 
\[
\int_{V_R^u(y)} \psi\,d\mu_y^u = (2Q_2e^{Q_u})^{\pm 2} \frac 1{m_y^\C(V_R^u(y))} \int_{V_R^u(y)} \psi\,dm_y^\C.
\]
Since $\psi>0$ was arbitrary, this proves \eqref{eqn:mu-m}, modulo some minor technical issues around the boundary of $\Vl^u(x)$ that are dealt with in \cite{CPZ}.

\subsubsection{Other properties of limiting measures}

Throughout this section, $\mu$ will denote an arbitrary $f$-invariant Borel probability measure on $\L$ that satisfies \eqref{eqn:mu-m}, so that the conditional measures $\mu_y^u$ are equivalent to the reference measures $m_y^\C$.  By the previous section, this includes every limit point of the sequence $\mu_n$.

We first observe that by Theorem \ref{thm:holonomy} and \eqref{eqn:mu-m}, for every rectangle $R$ with $\mu(R)>0$, $\mu$-a.e.\ $y,z\in R$, and every $A\subset V_R^u(z)$, we have
\begin{equation}\label{eqn:mu-cond-ac}
\mu_y^u(\pi_{zy} A) = Q_3^{\pm 1} m_y^\C(\pi_{zy} A) / m_y^\C(R)
= Q_3^{\pm 1} Q_2^{\pm 2} m_z^\C(A) / m_z^\C(R) = (Q_3 Q_2)^{\pm 2} \mu_z^u(A).
\end{equation}
In particular, holonomy maps along stable manifolds are absolutely continuous with respect to the conditional measures $\mu_y^u$, and thus by the standard Hopf argument (Proposition \ref{prop:hopf}), $\mu$ is ergodic.

Now we prove that $\mu$ is fully supported and satisfies the Gibbs property.  Let $\delta>0$ be small enough that for every $x\in \L$, the rectangle
\[
R(x,\delta) := [\overline{B_\L^u(x,\delta)},\overline{B_\L^s(x,\delta)}] = \{[y,z] : y\in \overline{B_\L^u(x,\delta)}, z\in \overline{B_\L^s(x,\delta)}\}
\]
is well-defined, as in \eqref{rectangle}.  Given $n\in \NN$, consider the rectangle
\[
R_n(x,\delta) := [\overline{B_n^u(x,\delta) \cap \L}, \overline{B_\L^s(x,\delta)}] \subset R(x,\delta).
\]
It is shown in \cite[Lemma 8.3]{CPZ} that for every $\delta>0$, there are $\delta_1,\delta_2>0$ such that
\begin{equation}\label{eqn:BnRn}
R_n(x,\delta_1) \subset B_n(x,\delta) \subset R_n(x,\delta_2)
\end{equation}
for every $x\in \L$ and $n\in \NN$; thus to prove the Gibbs property \eqref{gibbs1} it suffices to establish the corresponding bounds on $\mu(R_n(x,\delta))$.

\begin{lemma}\label{lem:pre-Gibbs}
Given $\delta>0$, there is $Q_5>0$ such that for every $x,\delta,n$ as above, we have
\begin{equation}\label{eqn:pre-Gibbs}
\mu(R_n(x,\delta)) = Q_5^{\pm 1} e^{-nP(\ph) + S_n\ph(x)} \mu(R(x,\delta)).
\end{equation}
\end{lemma}
\begin{proof}
Writing $\mu_y^u$ for the conditional measures of $\mu$ on unstable leaves in $R(z,\delta)$, we have
\begin{align*}
\mu(R_n&(x,\delta)) = \int_{R(x,\delta)} \mu_y^u(R_n(x,\delta)) \,d\mu(y)
= Q_3^{\pm 1} \int_{R(x,\delta)} \frac{m_y^\C(R_n(x,\delta))}{m_y^\C(R(x,\delta))} \,d\mu(y) \\
&= (KQ_3)^{\pm 1} \int_{R(x,\delta)} m_y^\C (\pi_{xy} \overline{B_n^u(x,\delta)}) \,d\mu(y)
= (KQ_3 Q_2)^{\pm 1} m_x^\C(\overline{B_n^u(x,\delta)}) \mu(R(x,\delta)),
%\\ &= (KQ_3 Q_2 Q_1)^{\pm 1} e^{-nP(\ph) + S_n\ph(x)} \mu(R(x,\delta)),
\end{align*}
where the first equality uses the definition of conditional measures, the second uses \eqref{eqn:mu-m}, the third uses Theorem \ref{thm:finite}, and the fourth uses Theorem \ref{thm:holonomy}.
%where the last estimate uses the $u$-Gibbs property of $m_x^\C$.
Since $B_n^u(x,\delta) \subset \overline{B_n^u(x,\delta)} \subset B_n^u(x,2\delta)$, the result follows from the $u$-Gibbs property of $m_x^\C$.
\end{proof}

\begin{lemma}[{\cite[Lemma 8.4]{CPZ}}]\label{lem:Rxz}
For every sufficiently small $\delta>0$, there is $\delta'>0$ such that for every  $z\in \L$ and $x\in R(z,\delta')$, we have $R(z,\delta') \subset R(x,\delta)$.
\end{lemma}

\begin{lemma}\label{lem:dense-orbit}
If $y\in \L$ has a backwards orbit that is dense in $\L$, then $\mu(R(y,\delta)) > 0$ for all $\delta>0$.
\end{lemma}
\begin{proof}
Let $\delta'>0$ be as in Lemma \ref{lem:Rxz}.  Since $\L$ is compact, there is a finite set $E\subset \L$ such that $\bigcup_{z\in E} R(z,\delta') = \L$, and thus there is $z\in E$ with $\mu(R(z,\delta'))>0$.  Since the backwards orbit of $y$ is dense, there is $n\geq 0$ such that $x := f^{-n}(y) \in R(z,\delta')$.  By Lemma \ref{lem:Rxz} and our choice of $x$, we have
\[
\mu(R(x,\delta)) \geq \mu(R(z,\delta')) > 0.
\]
By Lemma \ref{lem:pre-Gibbs}, we conclude that $\mu(R_n(x,\delta))>0$.  Moreover, we have
\[
f^n R_n(x,\delta) = f^n [B_n^u(x,\delta) \cap \L, B_\L^s(x,\delta)]
\subset [B_\L^u(y,\delta), B_\L^s(y,\delta)] = R(y,\delta),
\]
where we use the fact that $\|Df|_{E^s}\| \leq 1$.  Since $\mu$ is $f$-invariant, this gives $\mu(R(y,\delta)) \geq \mu(R_n(x,\delta))>0$.
\end{proof}

Since $f$ is topologically transitive on $\L$, every (relatively) open set in $\L$ contains a set of the form $R(y,\delta)$ where $y$ has a dense backwards orbit.  Thus Lemma \ref{lem:dense-orbit} implies that $\mu$ is fully supported on $\L$.

Finally, we deduce the Gibbs property \eqref{gibbs2} for $\mu$ as follows.  Given $\delta>0$, let $\delta'>0$ be as in Lemma \ref{lem:Rxz}, and once again let $E\subset \L$ be a finite set with $\bigcup_{z\in E} R(z,\delta') = \L$.  Since $\mu$ is fully supported, we have $\eta := \min_{z\in E} \mu(R(z,\delta')) > 0$.  Now given  any $x\in \L$, we have $x\in R(z,\delta')$ for some $z\in E$, and thus Lemma \ref{lem:Rxz} gives
\[
\mu(R(x,\delta)) \geq \mu(R(z,\delta')) \geq \eta.
\]
In particular, $\eta \leq \mu(R(x,\delta)) \leq 1$ for every $x\in \L$, and then the Gibbs property \eqref{gibbs2} follows immediately from Lemma \ref{lem:pre-Gibbs} and \eqref{eqn:BnRn}.

\subsubsection{Local product structure}\label{sec:lps}

To prove Corollary \ref{cor:prod}, we first observe that \eqref{eqn:mu-cond-ac}
gives $\pi^*_{yp}\mu^u_p \sim \mu^u_y$ for $\mu$-a.e.\ $p,y\in R$, which is the first claim.  For the second claim, we define a function $h\colon R\to (0,\infty)$ by $h(z) = \frac{d\mu^u_z}{d(\pi^*_{zp}\mu^u_p)}(z)$, so that \eqref{gibbs-split2} gives
\begin{multline}\label{eqn:otimes}
\mu(E)=\int_{V_R^s(p)}\int_{V_R^u(y)} \one_{E}(z)\,d\mu_y^u(z)\,d \tilde\mu_p(y) \\
= \int_{V_R^s(p)} \int_{V_R^u(y)} \one_E(z) h(z) \,d(\pi_{zp}^* \mu_p^u)(z) \,d\tilde\mu_p(y)
= \int_E h(z) \,d(\mu_p^u\otimes \tilde\mu_p)(z).
\end{multline}
for every measurable $E\subset R$.  For the third claim, we observe that \eqref{eqn:otimes} gives
\[
\mu(E) = \int_E h(z) \,d(\mu_p^u \otimes \tilde\mu_p)(z)
= \int_{V_R^u(p)} \int_{V_R^s(y)} \one_E(z) h(z) \,d(\pi_{yp}^* \tilde\mu_p)(z) \,d\mu_p^u(y),
\]
and since $\mu_y^s$ is uniquely determined up to a scalar (for $\mu$-a.e.\ $y$) by the condition that
\[
\mu(E) = \int_{V_R^u(p)} \int_{V_R^s(y)} \one_E(z) \,d\mu_y^s(z) \,d\nu(y)
\]
for some measure $\nu$ on $V_R^u(p)$, we conclude that $d\mu_y^s = h\,d(\pi_{yp}^*\tilde\mu_p)$.

\subsection{Proof of Theorem \ref{push-cond1}}\label{sec:push-cond}
We have that for any Borel subset $E\subset Y$
\begin{equation}\label{gibbs-split1}
\mu(E)=\int_{\tilde{Y}} \int_{W} \one_{E}(z)\,d\mu_W^\xi(z)\,d\tilde{\mu}(W).   
\end{equation}
Without loss of generality we may assume that $\mu_W^\xi$ is normalized, so that $\mu_W^\xi(W)=1$.
Consider the set $B$ of all Birkhoff generic points $x\in X$, for which
\begin{equation}\label{generic-p}
\lim_{n\to\infty}\frac{1}{n}\sum_{k=0}^{n-1}h(f^k(x))=\int_Xh\,d\mu
\end{equation}
for every continuous function $h$ on $X$. Since $\mu$ is ergodic we have that $B$ has full measure in $Y$. By \eqref{gibbs-split1}, there is a set $D\subset \tilde Y$ such that $\tilde\mu(\tilde Y\setminus D)=0$ and for every $W\in D$ we have $\mu_W^\xi(W\setminus B)=0$.  Given any such $W$ and any measure $\nu\ll \mu_W^\xi$, we have $\nu(X\setminus B)=0$.  Then Theorem \ref{push-cond1} is a consequence of the following general result.

\begin{proposition}\label{prop:generic-pts}
Let $X$ be a compact metric space, $f\colon X\to X$ a continuous map, and $\mu$ an $f$-invariant Borel probability measure on $X$.  Let $B$ be the set of Birkhoff generic points for $\mu$ and let $\nu$ be any probability measure on $X$ such that $\nu(B)=1$.  Then the sequence of measures $\nu_n := \frac 1n \sum_{k=0}^{n-1} f_*^k \nu$ converges in the weak* topology to the measure $\mu$.
\end{proposition}

Before proving Proposition \ref{prop:generic-pts}, we note that $\mu$ is not required to be ergodic.  In the case when $\mu$ is ergodic, Birkhoff's theorem gives $\mu(B)=1$, so that in particular $B$ is nonempty.  For non-ergodic $\mu$, the set $B$ can be either empty or nonempty.

\begin{proof}[Proof of Proposition \ref{prop:generic-pts}]
Let $\kappa$ be a weak* limit point of the sequence $\nu_n$, so that there is a subsequence $\{n_\ell\}_{\ell\in \NN}$ such that for every continuous function $h$ on $X$, we have
\begin{equation}\label{limit}
\int_Xh\,d\kappa=\lim_{\ell\to\infty}\int_Xh\,d\nu_{n_\ell}
=\lim_{\ell\to\infty}\int\frac{1}{n_\ell}\sum_{k=0}^{n_\ell-1}h\circ f^k d{\nu}. 
\end{equation}
We show that $\kappa \leq \mu$, which implies that $\kappa=\mu$ since both are probability measures.  It suffices to show that $\int h\,d\kappa \leq \int h\,d\mu$ for every nonnegative continuous function $h$.

Fix $h$ as above. Given $N\in\NN$ and $\eps>0$, let
\[
B_N(\eps) := \bigg\{x\in B : \Big| \frac 1n S_n h(x) - \int h\,d\mu \Big| < \eps \text{ for all } n\geq N\bigg\}.
\]
Then for every $\eps>0$ we have $\bigcup_{N\in \NN} B_N(\eps) = B$, hence there is $N_\eps$ such that $\nu(B\setminus B_N(\eps)) < \eps$.  By \eqref{limit} we can choose $n_\ell > N_\eps$ such that
\begin{align*}
\int_X h\,d\kappa &\leq \eps + \int \frac 1{n_\ell} S_{n_\ell} h\,d\nu 
= \eps + \int_{B_N(\eps)} \frac 1{n_\ell} S_{n_\ell} h\,d\nu
+ \int_{B\setminus B_N(\eps)} \frac 1{n_\ell} S_{n_\ell} h\,d\nu \\
&\leq 2\eps + \int h\,d\mu + \nu(B\setminus B_N(\eps)) \|h\|
< \eps(2+\|h\|) + \int h\,d\mu.
\end{align*}
Since $\eps>0$ was arbitrary this completes the proof of the proposition.
\end{proof}

\bibliographystyle{amsalpha}
\bibliography{car-ref}
\end{document}